\documentclass[12pt]{amsart}
\usepackage{amsfonts,cite,exscale}
\usepackage{bm}
\usepackage{mathrsfs}
\usepackage{txfonts}
\usepackage{amssymb}
\usepackage{graphicx, epsf,amssymb,amscd,amsfonts,times}
\usepackage{color,xspace,colortbl}
\usepackage{overpic}

\newcommand{\ba}{\boldsymbol\alpha}
\newcommand{\bb}{\boldsymbol\beta}
\newcommand{\xx}{\mathbf{x}}
\newcommand{\yy}{\mathbf{y}}
\newcommand{\TT}{\mathbb{T}}
\newcommand{\PP}{\mathcal{P}}

\newcommand{\Z}{\mathbb{Z}}

\newcommand{\fs}{\mathfrak{s}}
\newcommand{\R}{\mathbb{R}}
\newcommand{\C}{\mathbb{C}}
\newcommand{\spinc}{Spin\textsuperscript{c} }

\newcommand{\s}{\vskip.1in}
\newcommand{\n}{\noindent}
\newcommand{\bdry}{\partial}
\newcommand{\bdsym}{\boldsymbol}

\newcommand{\be}{\begin{enumerate}}
\newcommand{\ee}{\end{enumerate}}

\newtheorem{thm}{Theorem}[section]

\newtheorem{lemma}[thm]{Lemma}
\newtheorem{cor}[thm]{Corollary}

\theoremstyle{remark}
\newtheorem{rmk}[thm]{Remark}

\numberwithin{equation}{subsection}

\begin{document}

\title[An absolute grading on Heegaard Floer homology]{An absolute grading on Heegaard Floer homology by homotopy classes of oriented 2-plane fields}

\author{Vinicius Gripp}
\address{UC Berkeley, Berkeley, CA 94720\\USA}
\email{vinicius@math.berkeley.edu}

\author{Yang Huang}
\address{University of Southern California, Los Angeles, CA 90089\\USA}
\email{huangyan@usc.edu}

\begin{abstract}
For a closed oriented 3-manifold $Y$, we define an absolute grading on the Heegaard Floer homology groups of $Y$ by homotopy classes of oriented 2-plane fields. We show that this absolute grading refines the relative one and that it is compatible with the maps induced by cobordisms. We also prove that if $\xi$ is a contact structure on $Y$, then the grading of the contact invariant $c(\xi)$ is the homotopy class of $\xi$.
\end{abstract}

\maketitle

\section{Introduction}

For a closed oriented 3-manifold $Y$, Ozsv\'ath and Szab\'o~\cite{OSz2} defined a collection of invariants of $Y$,  the Heegaard Floer homology groups $HF^\circ(Y)$, where $HF^\circ(Y)$ denotes either $\widehat{HF}(Y)$, $HF^+(Y)$, $HF^-(Y)$, or $HF^\infty(Y)$. They showed that $HF^\circ(Y)$ splits into a direct sum by Spin$^c$ structures
\begin{equation*}
    HF^\circ(Y)=\bigoplus_{\mathfrak{s}\in\textrm{Spin}^c(Y)} HF^\circ(Y,\mathfrak{s}).
\end{equation*}
For each $\mathfrak{s}\in\textrm{Spin}^c(Y)$, they also defined a relative grading on $HF^\circ(Y,\mathfrak{s})$, that takes values in $\mathbb{Z}/d(c_1(\mathfrak{s}))$, where $d(c_1(\mathfrak{s}))$ is the divisibility of $c_1(\mathfrak{s})\in H^2(Y;\Z)$, i.e. $d(c_1(\mathfrak{s}))\Z=\langle c_1(\mathfrak{s}), H_2(Y)\rangle.$

Moreover given a 4-dimensional compact oriented cobordism $W:Y_0 \to Y_1$, i.e. $\bdry W=-Y_0 \cup Y_1$ as oriented manifolds, and given a Spin$^c$ structure $\mathfrak{t}$ on $W$, there is a natural map $F_{W,\mathfrak{t}}:HF^\circ(Y_0,\mathfrak{t}|_{Y_0}) \to HF^\circ(Y_1,\mathfrak{t}|_{Y_1})$  defined by Ozsv\'ath-Szab\'o~\cite{OSz3}.

It has been shown that Heegaard Floer homology is isomorphic to two other homology theories: Seiberg-Witten Floer homology~\cite{KM} and embedded contact homology (ECH)~\cite{Hut1,Hut3,Hut4}. For a proof of the existence of these isomorphisms, see~\cite{T,KLT,CGH}. It is known that both ECH~\cite{Hut2} and Seiberg-Witten Floer homology~\cite{KM} are absolutely graded by homotopy classes of oriented 2-plane fields, but no such absolute grading had been defined for Heegaard Floer homology. In this paper, we construct such an absolute grading for Heegaard Floer homology, which is compatible with the relative grading and cobordism maps discussed above.

We will now fix some notation that will be used in this paper. Let $(\Sigma,\bm{\alpha},\bm{\beta},z)$ be a Heegaard diagram of $Y$. Here $\Sigma$ is a genus $g$ surface, $\ba=(\alpha_1,\dots,\alpha_g)$ and $\bb=(\beta_1,\dots,\beta_g)$ are collections of disjoint circles on $\Sigma$ and the basepoint $z$ is a point on $\Sigma$ in the complement of $\alpha_1\cup\dots\cup\alpha_g\cup\beta_1\cup\dots\cup\beta_g$. We also require that $\ba$ and $\bb$ are linearly independent sets in $H_1(Y)$ and that $\alpha_i$ and $\beta_j$ intersect transversely for every $i$ and $j$. We consider the tori $\mathbb{T}_ \alpha=\alpha_1\times\dots\times\alpha_g$ and $\mathbb{T}_{\beta}=\beta_1\times\dots\times\beta_g$ in the symmetric product $\text{Sym}^g(\Sigma)$.
Recall that the Heegaard Floer chain complex $\widehat{CF}(Y)$ is the free abelian group generated by the intersection points $\mathbf{x}\in \mathbb{T}_\alpha \cap \mathbb{T}_\beta$.  If $\xx$ and $\yy$ are intersection points in the same \spinc structure, we denote by $\text{gr}(\xx,\yy)$ their relative grading, as defined in~\cite{OSz2}.

We denote by $\mathcal{P}(Y)$ the set of homotopy classes of oriented 2-plane fields on $Y$. Each homotopy class of oriented 2-plane fields belongs to a \spinc structure, as we will explain in Section~\ref{SecDef}. Therefore $\mathcal{P}(Y)$ splits by Spin$^c$ structures as
$$\mathcal{P}(Y)=\coprod_{\fs\in\text{Spin}^c(Y)}\mathcal{P}(Y,\fs).$$
It turns out that $\mathcal{P}(Y,\fs)$ is an affine space over $\Z/d(c_1(\mathfrak{s}))$. For each Spin$^c$ structure $\fs$, we will construct an absolute grading $\widetilde{\text{gr}}$ on $\widehat{CF}(Y,\fs)$ with values in $\mathcal{P}(Y,{\fs})$.

For a contact structure $\xi$ on $Y$, Ozsv\'ath-Szab\'o~\cite{OSz1} defined the contact invariant $c(\xi)\in\widehat{HF}(-Y)$.
In~\cite{OSz2}, Ozsv\'ath-Szab\'o showed that a Heegaard move induces an isomorphism on Heegaard Floer homology.

Consider a compact oriented cobordism $W:Y_0 \to Y_1$. Let $\xi_0$ and $\xi_1$ be oriented 2-plane fields on $Y_0$ and $Y_1$ respectively. We say that $\xi_0 \sim_W \xi_1$ if there exists an almost complex structure $J$ on $W$ such that $[\xi_0]=[TY_0 \cap J(TY_0)]$ and $[\xi_1]=[TY_1 \cap J(TY_1)]$ as homotopy classes of oriented 2-plane fields.

We can now state the main theorem of this paper.

\begin{thm} \label{mainThm}
For every Heegaard diagram $(\Sigma,\ba,\bb,z)$ of $Y$, there exists a canonical function $\widetilde{gr}: \TT_{\alpha}\cap\TT_{\beta} \to \mathcal{P}(Y)$ such that:
\begin{enumerate}
    \item[(a)]{If $\mathbf{x},\mathbf{y} \in\TT_{\alpha}\cap\TT_{\beta} $ are in the same Spin$^c$ structure $\fs$, then $\widetilde{\text{gr}}(\mathbf{x})$ and $\widetilde{\text{gr}}(\mathbf{y})$ belong to $\mathcal{P}(Y,{\fs})$ and $\widetilde{\text{gr}}(\mathbf{x})-\widetilde{\text{gr}}(\mathbf{y})=\text{gr}(\xx,\yy)\in\Z/d(c_1(\fs))$. In particular, $\widetilde{\text{gr}}$ extends to the set of homogeneous elements of $\widehat{CF}(Y)$.}
    \item[(b)]{Let $\xi$ be a contact structure on $Y$, and let $c(\xi) \in \widehat{HF}(-Y)$ be the contact invariant. Then $\widetilde{\text{gr}}(c(\xi))=[\xi]$ as homotopy classes of oriented 2-plane fields.}
    \item[(c)] This absolute grading is invariant under the isomorphisms induced by Heegaard moves and hence it induces an absolute grading on $\widehat{HF}(Y)$ which is independent of the Heegaard diagram.
    \item[(d)]{Let $W: Y_0 \to Y_1$ be a compact, oriented cobordism, and let $\mathfrak{t}$ be a Spin$^c$ structure on $W$. Then the induced map $F_{W,\mathfrak{t}}: \widehat{HF}(Y_0,\mathfrak{t}|_{Y_0}) \to \widehat{HF}(Y_1,\mathfrak{t}|_{Y_1})$ respects the grading in the sense that $\widetilde{\text{gr}}(\mathbf{x}) \sim_W \widetilde{\text{gr}}(\mathbf{y})$ for any homogeneous element $\mathbf{x} \in \widehat{HF}(Y_0,\mathfrak{t}|_{Y_0})$ and any $\mathbf{y} \in \widehat{HF}(Y_1,\mathfrak{t}|_{Y_1})$, which is a homogeneous summand of $F_{W,\mathfrak{t}}(\mathbf{x})$.}
\end{enumerate}
\end{thm}

\begin{rmk}
Theorem~\ref{mainThm}(a) implies that we have the following decomposition by degrees.
\begin{equation}\label{eqdecomp}
 \widehat{CF}(Y;\fs)=\bigoplus_{\rho\in\mathcal{P}(Y,\fs)} \widehat{CF}_{\rho}(Y;\fs).
\end{equation}
Here $\widehat{CF}_{\rho}(Y;\fs)$ is the $\Z$-module generated by all $\xx\in\TT_{\ba}\cap\TT_{\bb}$ with $\widetilde{\text{gr}}(\xx)=\rho$.
\end{rmk}

\begin{rmk}\label{RmkAct}
The generators of $HF^{\infty}(Y)$ are of the form $[\xx,i]$, where $\xx\in\TT_{\alpha}\cap\TT_{\beta}$ and $i\in\Z$. We recall that $\Z$ acts on $\mathcal{P}(Y)$, since $\mathcal{P}(Y,\fs)$ is an affine space over $\Z/d(c_1(\fs))$. So we can define an absolute grading on $HF^{\infty}(Y)$, and hence on $HF^{-}(Y)$ and $HF^{+}(Y)$, by $\widetilde{\text{gr}}([\xx,i])=\widetilde{\text{gr}}(\xx)+2i$, for a homogeneous element $\xx$. It is easy to see that Theorem~\ref{mainThm} implies that (a),(c) and (d) also hold for $HF^{\infty}(Y)$, $HF^{-}(Y)$ and $HF^{+}(Y)$.
\end{rmk}

\begin{rmk}
Using the absolute grading function $\widetilde{\text{gr}}$ constructed in Theorem~\ref{mainThm}, one can recover the absolute $\mathbb{Q}$-grading for $HF^\circ(Y,\mathfrak{s})$ defined by Ozsv\'ath-Szab\'o when $c_1(\mathfrak{s}) \in H^2(Y;\Z)$ is a torsion class. See Corollary~\ref{ColQ} for details.
\end{rmk}

We can also generalize the absolute grading function $\widetilde{gr}$ to the twisted Heegaard Floer homology groups defined by Ozsv\'ath-Szab\'o~\cite{OSz5}. Recall that the twisted Heegaard Floer homology group $\underline{HF}(Y,\mathfrak{s})$ is the homology of the twisted Heegaard Floer chain complex $CF(Y;\mathfrak{s})\otimes\mathbb{Z}[H^1(Y;\mathbb{Z})]$, where the (infinity version) differential is defined by
\begin{equation*}
    \underline{\bdry}^\infty[\mathbf{x},i]=\sum_{\mathbf{y}\in\mathbb{T}_\alpha\cap\mathbb{T}_\beta} \Big(\sum_{\phi\in\pi_2(\mathbf{x},\mathbf{y})} \#\mathcal{M}(\phi)e^{A(\phi)}[\mathbf{y},i-n_z(\phi)] \Big)
\end{equation*}
where $A:\pi_2(\mathbf{x},\mathbf{y}) \to H^1(Y;\mathbb{Z})$ is a surjective, additive assignment. See~\cite{OSz5}  for more details. Now we define the twisted absolute grading function by simply ignoring the twisted coefficient as follows:
\begin{eqnarray} \label{twgr}
    \widetilde{gr}_{tw}: \mathbb{Z}[H^1(Y;\mathbb{Z})](\mathbb{T}_\alpha \cap \mathbb{T}_\beta) \to & \mathcal{P}(Y)\\
    e^{\xi}\mathbf{x} \mapsto &\widetilde{gr}(\mathbf{x}),\nonumber
\end{eqnarray}
where $\xi \in H^1(Y;\mathbb{Z})$ and we write $\mathbb{Z}[H^1(Y;\mathbb{Z})]$ multiplicatively.\footnote{The twisted absolute grading defined here does not refine the relative $\Z$-grading within each Spin$^c$ structure defined in \cite{OSz5}. A slightly more sophisticated construction of the twisted grading is needed to recover the relative $\Z$-grading. But since we do not need this refinement in this paper, we do not include the details here.} Using an obvious twisted version of Theorem~\ref{mainThm}(b), we will prove the following corollaries in Section~\ref{SecOnContCls}.

Let $\mathcal{F}_Y$ denote the set of homotopy classes (as 2-plane fields) of contact structures on $Y$ which are weakly fillable.

\begin{cor}[Kronheimer-Mrowka \cite{KM97}] \label{KM97}
$\mathcal{F}_Y$ is finite.
\end{cor}

\begin{cor} \label{L-sp}
If $Y$ is an $L$-space, then $|\mathcal{F}_Y| \leq |H_1(Y;\mathbb{Z})|$.
\end{cor}

\begin{cor}[Lisca \cite{Lis}] \label{Lis}
If $Y$ admits a metric of constant positive curvature, then $|\mathcal{F}_Y| \leq |H_1(Y;\mathbb{Z})|$.
\end{cor}

\begin{rmk}
Corollary~\ref{KM97} and Corollary~\ref{Lis} are previously proved using the relationship between Seiberg-Witten theory and contact topology. 
\end{rmk}

\begin{rmk}
In fact the assertion in Corollary~\ref{KM97} holds for the set of homotopy classes of 2-plane fields which support a tight contact structure by the work of Colin-Giroux-Honda \cite{CGirH}. But our result does not imply this generalization. In particular we do not have an upper bound on $|\mathcal{F}(Y)|$ for tight contact structures.
\end{rmk}

The paper is organized as follows. In Section 2, we construct the absolute grading on $\widehat{CF}$, which refines the relative grading defined in~\cite{OSz2}. That proves part (a) of the Theorem.
In Section 3, we compute the absolute grading of the contact invariant and show that it is the homotopy class of the contact structure, which proves part (b) of the Theorem. This fact is known, by construction, for the absolute grading in ECH~\cite{Hut2}. In Section 4, we prove part (d) at the chain level, showing that $\widetilde{\text{gr}}$ is natural under cobordism maps, as stated in Theorem~\ref{CobThm}. This was shown for Seiberg-Witten Floer homology by Kronheimer-Mrowka~\cite{KM}. In Section~5, we prove that $\widetilde{\text{gr}}$ is preserved under Heegaard moves, see Theorem \ref{inv}. That means that the decomposition (\ref{eqdecomp}) is preserved under Heegaard moves and therefore it also holds in the homology level. That implies that part (d) also holds in homology.\\

\noindent
{\em Acknowledgements}. We would like to thank Ko Honda and Michael Hutchings for suggesting this problem to us and for providing guidance throughout the course of this project. We also thank Tye Lidman for pointing out applications of the absolute grading of the contact invariant to us. This work started during our visit to the Mathematical Sciences Research Institute in 2009-2010, where an excellent environment for math research was provided. The first author was partially supported by NSF grant DMS-0806037.

\section{The absolute grading}\label{SecDef}

Let $Y$ be an oriented closed 3-manifold and let $\PP(Y)$ denote the set of homotopy classes of oriented 2-plane fields on $Y$. Let us first recall that there is a surjection  $\psi:\PP(Y)\rightarrow\text{Spin}^c(Y)$. Also, for a fixed \spinc structure $\fs$, we can endow $\psi^{-1}(\fs)=\PP(Y,\fs)$ with the structure of an affine space over $\Z/d(c_1(\fs))$, where $d(c_1(\fs))$ is the divisibility of the first Chern class of $\fs$. So, given $\xi,\eta\in\PP(Y)$ mapping to the same \spinc structure $\mathfrak{s}$, there is a well-defined difference $\xi-\eta$. One way of seeing this affine space structure is by using the Pontryagin-Thom construction, as follows. Each $\xi\in\PP(Y)$ corresponds to a unique homotopy class of nonvanishing vector fields, which we denote by $[v_{\xi}]$. Fixing a representative $v_{\xi}$ and a trivialization of $TY$, and after a normalization, we can think of $v_{\xi}$ as a map $Y\rightarrow S^2$. The preimage of a regular value of this map gives a link and the preimage of the tangent plane to this regular point under the derivative map determines a framing of this link. We recall that two framed links $L_O, L_1\subset Y$ are called framed cobordant, if there exists a framed surface $S\subset Y\times[0,1]$, whose boundary is $-L_O\times\{0\}\cup L_1\times \{1\}$ and such that the framing restricted to the boundary coincides with the initial framings on $L_0$ and $L_1$. It follows from Pontryagin-Thom theory that two nonvanishing vector fields are homotopic if and only if the respective framed links are framed cobordant. If $\xi,\eta$ map to the same \spinc structure, then the respective links are cobordant and the difference of framings is $\xi-\eta\in\Z/d(c_1(\fs))$. The sign convention we are using here is that a left-handed twist increases a framing by $+1$.

Now let $(\Sigma,\ba,\bb,z)$ be a Heegaard diagram representing $Y$, where $\ba=(\alpha_1,\dots,\alpha_g)$ and $\bb=(\beta_1,\dots,\beta_g)$. Recall that the generators of $\widehat{CF}(Y)$ are the intersection points of the tori $\TT_{\alpha}$ and $\TT_{\beta}$ in  $\text{Sym}^g(\Sigma)$. Our goal in this section is to construct a canonical map $ \TT_{\alpha}\cap\TT_{\beta}\rightarrow \PP(Y)$ that refines the relative grading, which we denote by \emph{gr}, and the map that assigns a \spinc structure to a generator, which we denote by $\fs_z:\TT_{\alpha}\cap\TT_{\beta}\rightarrow\text{Spin}^c(Y)$. For the definitions of these maps, see~\cite{OSz2}.

\begin{thm}\label{ThmRel}
There is a canonical map $\widetilde{\text{gr}}:\TT_{\alpha}\cap\TT_{\beta}\rightarrow \PP(Y)$, such that if $\xx,\yy\in\TT_{\ba}\cap\TT_{\bb}$ are such that $\fs_z(\xx)=\fs_z(\yy)=\fs$, then $$\widetilde{\text{gr}}(\xx)-\widetilde{\text{gr}}(\yy)=\text{gr}(\xx,\yy)\in \Z/d(c_1(\fs)).$$
\end{thm}

\subsection{The construction}\label{Defgr}

We fix a self-indexing Morse function $f:Y\rightarrow \R$ compatible with $(\Sigma,\ba,\bb)$. Let $\xx\in\TT_{\alpha}\cap\TT_{\beta}$. Then $\xx$ corresponds to $g$ points $x_1,\dots,x_g$ on $\Sigma$, which give rise to flow lines $\gamma_{x_1},\dots,\gamma_{x_g}$ connecting the index 1 critical points to the index 2 critical points. The basepoint $z$ determines a flow line $\gamma_0$ from the index 0 critical point to the index 3 critical point. We can choose a gradient-like vector field $v$, tubular neighborhoods $N(\gamma_{x_i})$ of $\gamma_{x_i}$ and diffeomorphisms $N(\gamma_{x_i})\cong B^3$ such that, under these diffeomorphisms,
$v_{\vert N(\gamma_{x_i})}:B^3\rightarrow \R^3$ is given by $v(x,y,z)=(x,-y,1-2z^2)$, for $i\neq 0$ and $v_{\vert N(\gamma_0)}:B^3\rightarrow \R^3$ is given by
$v(x,y,z)=(2xz,2yz,1-2z^2)$. Figure \ref{nbhd}(a) shows two cross-sections of $v_{\vert N(\gamma_{x_i})}$, for $i\neq0$. Figure \ref{nbhd}(b) shows $v_{\vert N(\gamma_0)}$ on any plane passing through the origin containing the $z$-axis.
Outside the union of the neighborhoods $N(\gamma_{x_i})$, $v$ is a nonvanishing vector field. We will define a nonvanishing continuous vector field $w_{\xx}$ on $Y$ that coincides with $v$ in the complement of the neighborhoods $N(\gamma_{x_i})$.

\begin{figure}[ht]
    \begin{overpic}[scale=.33]{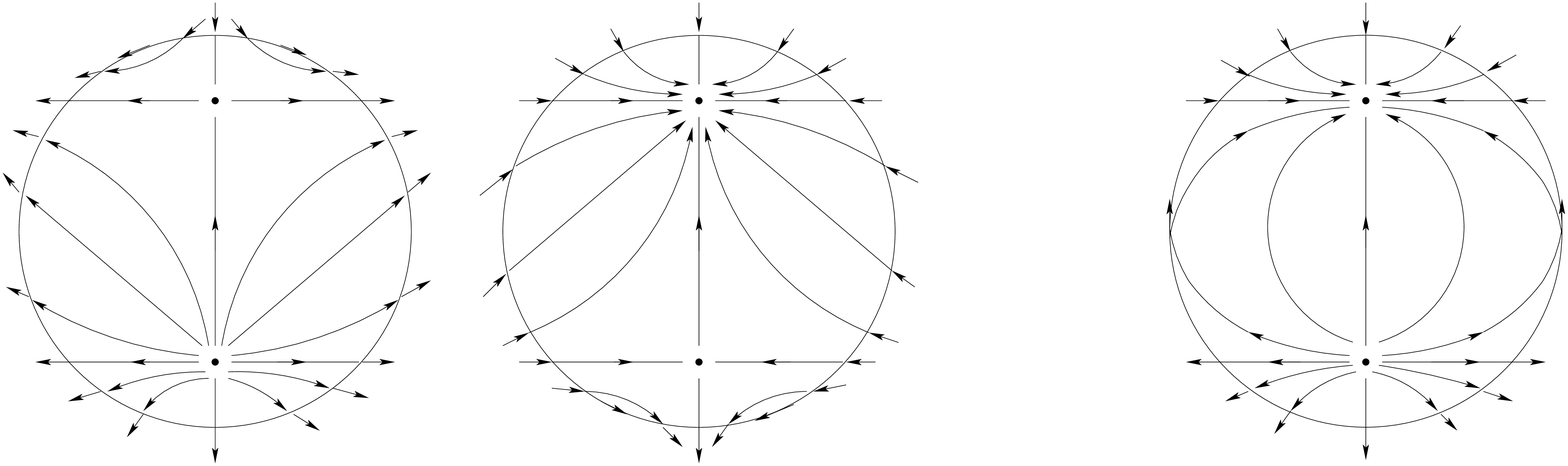}
    \put(10,-3){$xz$-plane}
    \put(40,-3){$yz$-plane}
    \put(27.5,-7){(a)}
    \put(86,-7){(b)}
    \end{overpic}
    \newline
    \newline
    \caption{}
    \label{nbhd}
\end{figure}

For $i\neq 0$, on $\partial N(\gamma_{x_i})\cong \partial B^3$, we note that
$$v(x,y,z)=(x,-y,1-2z^2)=(x,-y,2x^2+2y^2-1).$$
We define $w_{\xx}=(x,-y,2x^2+2y^2-1)$ in $N(\gamma_i)$, see Fig \ref{defn}(a).
This is a nonzero vector field in $N(\gamma_{x_i})$ that coincides with $v$ on $\partial N(\gamma_{x_i})$.
Also, on $\partial N(\gamma_0)$, we see that
$$v(x,y,z)=(-2xz,-2yz,1-2z^2)=(-2xz,-2yz,2x^2+2y^2-1).$$
This new vector field is still zero on the circle $C=\{(x,y,z)|x^2+y^2=1/2,z=0\}$. A vertical section of it in $B^3$ is shown in Figure \ref{defn}(b).So we define $w_{\xx}$ in $N(\gamma_0)$ by $$w_{\xx}(x,y,z)=(-2xz,-2yz,2x^2+2y^2-1)+\phi(x,y,z)(y,-x,0),$$ where $\phi$ is a bump function around $C$ (i.e. $\phi=1$ on $C$ and $\phi=0$ in the complement of a small neighborhood of $C$).
Therefore $w_{\xx}$ is a nonvanishing vector field on $Y$ that equals $v$ outside the union of the neighborhoods $N(\gamma_{x_i})$. We can perturb $w_{\xx}$ to a smooth vector field. Finally we define $\widetilde{\text{gr}}(\xx)$ to be the homotopy class of the orthogonal complement of $w_{\xx}$.

\begin{figure}[ht]
    \begin{overpic}[scale=.34]{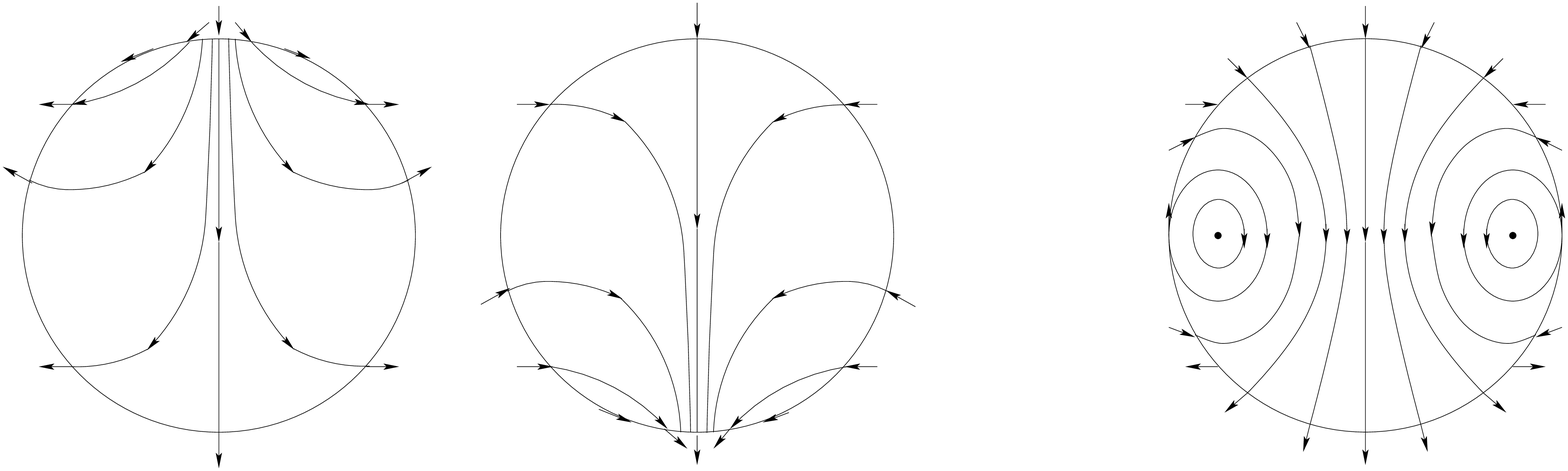}
    \put(10,-3){$xz$-plane}
    \put(40,-3){$yz$-plane}
    \put(27.5,-7){(a)}
    \put(86,-7){(b)}
    \end{overpic}
    \newline
    \newline
    \caption{}
    \label{defn}
\end{figure}

\begin{rmk}
We could use the gradient vector field itself instead of some other gradient-like vector field to define the absolute grading, but it would be harder to write down the formulas for the canonical modification of the gradient vector field in the neighborhoods of the flow lines. Nevertheless, we would obtain the same homotopy class.
\end{rmk}

\subsection{The relative grading}

This subsection is dedicated to proving that the absolute grading refines the relative grading. Given two intersection points $\xx,\yy\in\TT_{\ba}\cap\TT_{\bb}$ such that $\fs_z(\xx)=\fs_z(\yy)$, there exists a Whitney disk $A\in \pi_2(x,y)$, as proven in \cite{OSz2}. This means that $A$ is a homotopy class of maps $\varphi:D^2\subset \C\rightarrow\text{Sym}^g(\Sigma)$ taking $i$ to $\xx$, $-i$ to $\yy$, the semicircle with positive real part to $\TT_{\beta}$ and the one with negative real part to $\TT_{\alpha}$. Let $D_1,\dots,D_n$ denote the closures of the connected components of $\Sigma-\alpha_1-\dots-\alpha_g-\beta_1-\dots-\beta_g$. We write $D(A)=\sum_{k=1}^n a_k D_k$, where $a_k$ is the multiplicity of $\varphi$ on each $D_k$. We can choose a Whitney disk $A$ so that $a_k \ge 0$ for every $k$.

We will now construct surfaces $F_1\supset\dots\supset F_m$, whose union projects to $\sum_{k=1}^n a_k D_k=D(A)$ on $\Sigma$. We take $a_k$ copies of each $D_k$ and we glue them along their boundaries in the following way: we construct $F_1$ by gluing one copy of each $D_k$ with $a_k>0$. Then we construct $F_2$ by gluing one copy of each $D_k$ such that $a_k-1>0$. Inductively we construct surfaces $F_1,\dots,F_m$, where $m=\max{a_k}$. So the union of the surfaces $F_l$ can be identified with $D(A)$.
(Similar constructions can be found in~\cite{OSz2,Lip,R}).

The Euler measure of a surface with corners $S$, denoted by $e(S)$, is defined to be $\chi(S)-\frac{p}{4}+\frac{q}{4}$, where $p$ is the number of convex corners of $S$ and $q$ is the number of concave corners of $S$.
If $w\in\alpha_i\cap\beta_j$, for some $i,j$, then a small neighborhood of $w$, when intersected with the complement of the union of the $\alpha$ and the $\beta$ curves, gives rise to four regions. We define $n_{w}(D_k)$ to be $1/4$ times the number of those regions contained in $D_k$. We extend $n_w$ linearly to the $\Z$-module generated by the domains $D_k$. Now we define $n_{\xx}$ to be the sum of all $n_{x_i}$, for $i=1,\dots,g$. For example, a convex corner $x_i$ of $F_l$ contributes to $n_{\xx}(F_l)$ with $1/4$ and a concave corner $x_i$ with $3/4$. Similarly we define $n_{\yy}$.
By Lipshitz~\cite{Lip}, the Maslov index of the Whitney disk $A$, denoted by $\mu(A)$, is given by $$\mu(A)=\text{ind}(A)=e(D(A))+n_{\xx}(D(A))+n_{\yy}(D(A))=\sum_{l=1}^m \Big(e(F_l)+n_{\xx}(F_l)+n_{\yy}(F_l)\Big).$$
For each $D_k$, we define $n_z(D_k)$ to be $0$ if $z\not\in D_k$ and $1$ if $z\in D_k$, and we extend $n_z$ linearly to sums of $D_k$. The relative grading was defined by Ozsv\'{a}th-Szab\'{o}~\cite{OSz2} to be $$\text{gr}(\xx,\yy)=\mu(A)-2n_z(D(A))\in\Z/d,$$ where $d$ is the divisibility of $c_1(\fs(\xx))$. So we need to show that $$\widetilde{\text{gr}}(\xx)-\widetilde{\text{gr}}(\yy)= \sum_{l=1}^m \Big(e(F_l)+n_{\xx}(F_l)+n_{\yy}(F_l)-2n_z(F_l)\Big)\in\Z/d.$$

\textit{Step 1}: We first assume that $m=1$ and that $n_{z}(F_1)=0$. Recall that a corner $x_i$ is called degenerate if $x_i=y_j$ for some $j$. We also assume that there are no degenerate corners.

We will now choose a convenient trivialization of $TY$ in order to apply the Pontryagin-Thom construction. Let $f$ be a self-indexing Morse function $f$, which is compatible with $(\Sigma,\ba,\bb)$. Let $F:=F_1$.
Let $p_i$ be the index 1 critical point corresponding to $\alpha_i$ and $q_j$ the index 2 critical point corresponding to $\beta_j$. Each edge of the boundary of $F$ is part of an $\alpha_i$ or a $\beta_j$. So each edge of $\partial F$ determines a surface by flowing downwards or upwards towards a $p_i$ or $q_j$, respectively, and, by adding $p_i$ and $q_j$, we get a compact surface with corners. This surface has typically three corners unless it corresponds to an edge starting at a boundary degenerate corner in which case, this edge is actually a circle and the surface corresponding to it is a disk. We call $A_i$ and $B_j$ the surfaces corresponding to the edges contained in $\alpha_i$ and $\beta_j$, respectively. We note that the flow we consider here is the one generated by a gradient-like vector field $v$ compatible with the Morse function $f$.

Let $C$ be the union of $F$ and the surfaces $A_i$ and $B_j$. We will first choose a trivialization of $TY$ on $C$.
We start by defining a unit vector field $E_1$, which is tangent to $F$. The orientation of $\Sigma$ induces an orientation on $F$. We set $E_1$ to be the positive unit tangent vector along $\partial F$, with respect to its boundary orientation, outside a small neighborhood of the corners. At a neighborhood of a corner, we define $E_1$ on $\partial F$ by keeping it tangent to $F$ and rotating it by the smallest possible angle. That means that once we start rotating, $E_1$ will not be tangent to $\partial F$ at any point. In other words, each connected component of the set of points of $\partial F$ at which $E_1$ is not tangent to $\partial F$ contains exactly one corner of $F$.
We also have to choose a corner to rotate an extra $2\pi\chi(F)$ clockwise. That allows us to extend $E_1$ to $F$. We now define $E_1$ on each $A_i$ and $B_j$ to be an extension of $E_1$ on $\partial F$ such that it is tangent to $A_i$ and $B_j$ everywhere outside small neighborhoods of the corners $x_i$ and $y_j$ and such that it is always transverse to the flow lines $\gamma_{x_i}$ and $\gamma_{y_j}$. In particular $E_1$ is tangent to $A_i$ near $p_i$ and to $B_j$ near $q_j$. Near the corners $x_i$ and $y_j$, we require $E_1$ to never be tangent to $A_i$ and $B_j$, similarly to how we defined $E_1$ on $F$.
%We note that $E_1$ is well defined on $\partial C$ up to homotopy through nonzero vector fields and once we fix it at $\partial C$, it is well-defined up to homotopy relative to the boundary.
We define $E_3$ on $F$ to be the positive normal vector field to $F$, and we extend it to $A_i$ and $B_j$ so that $\{E_1,E_3\}$ is an oriented orthonormal frame on the respective tangent spaces, except maybe outside a small neighborhood of $\partial F$. In this neighborhood, we require that each connected component of the set of points where $E_3$ is not tangent to $A_i$ or $B_j$ intersects $F$.
%It is clear that $E_3$ is well-defined up to homotopy as before.
Now we take $E_2$ to be the unit vector field on $C$ orthogonal to $E_1$ and $E_3$ such that $\{E_1,E_2,E_3\}$ is an oriented basis of $TY$.
So mapping $E_i$ to $e_i\in \R^3$, we get a trivialization of $TY$ along $C$. We extend this trivialization to a neighborhood of $C$ in such a way that $E_1$ and $E_3$ are still tangent to the corresponding unstable and stable surfaces near the critical points $p_i$ and $q_j$ and that $e_1$ is a regular value of $w_{\xx}$ and $w_{\yy}$ when seen as maps $Y\rightarrow S^2$. Now, since there are no degenerate points, $C$ does not contain an $\alpha$ or $\beta$ curve. Therefore there is no obstruction to extending this trivialization to all of $Y$. So we choose one of those extensions.

Now we define $K_{\xx}'=w_{\xx}^{-1}(e_1)$ and $K_{\yy}'=w_{\yy}^{-1}(e_1)$ as framed links. We note that inside neighborhoods of the flow lines $\gamma_{x_i}$ and $\gamma_{y_i}$, these are one stranded braids contained in the corresponding unstable or stable surface, except that near each corner of $F$, this braid rotates around the respective flow line as much as $E_1$ restricted to this flow line does, but in the opposite direction. This is shown in Figure \ref{cobord}(a). It follows from the way that we chose the trivialization on $C$ that $K_{\xx}'$ and $K_{\yy}'$ do not intersect $C$ outside of those neighborhoods.

We can isotope $K_{\xx}'$ in neighborhoods of each $\gamma_{x_i}$ in the following way. Near each corner, this link is rotating around $\gamma_{x_i}$. We isotope a neighborhood of this part of the link to the segment of the flow line about which it is rotating fixing the endpoints. Outside of this neighborhood of the corner, but still inside the neighborhood of the flow line, the link is contained in the corresponding unstable or stable surface. We will call this new link $K_{\xx}$. We can think of the framing of a link as a unit normal vector field to the link. So the framing on $K_{\xx}$ induced from this isotopy can be seen by a vector field that is normal to the stable and unstable surfaces away from the corners and rotates with respect to the stable surface as much as $K'_{\xx}$ rotates about the flow line, as seen in Figure \ref{cobord}(b). We denote this framing by $\tau_{\xx}$. We note that once we fix which of the two unit normal vector fields to the stable surface we choose, the unit normal vector field to the unstable surface is determined.

We can do the same for $K_{\yy}'$ and define $K_{\yy}$ with framing denoted by $\eta_{\yy}$. Figure \ref{cobord}(c) shows a picture of both $K_{\xx}$ and $K_{\yy}$ at a neighborhood of a flow line $\gamma_{x_i}$. Now we modify $C$ in the following way. For each edge of $F$, we substitute the corresponding $A_i$ or $B_j$ by the region on the unstable or stable surface bounded by the corresponding edge of $F$ and the segments of $K_{\xx}$ and $K_{\yy}$, see Figure \ref{cobord}(c). We smooth the edges of this surface and denote by $\tilde{C}$ this smooth surface with boundary, which has cusps. We note that $\tilde{C}$ gives rise to a cobordism $S\subset Y\times[0,1]$ between $K_{\xx}\times\{0\}$ and $K_{\yy}\times\{1\}$ that is trivial where $K_{\xx}$ and $K_{\yy}$ coincide.

\begin{center}\begin{figure}[ht]
    \begin{overpic}[scale=0.4]{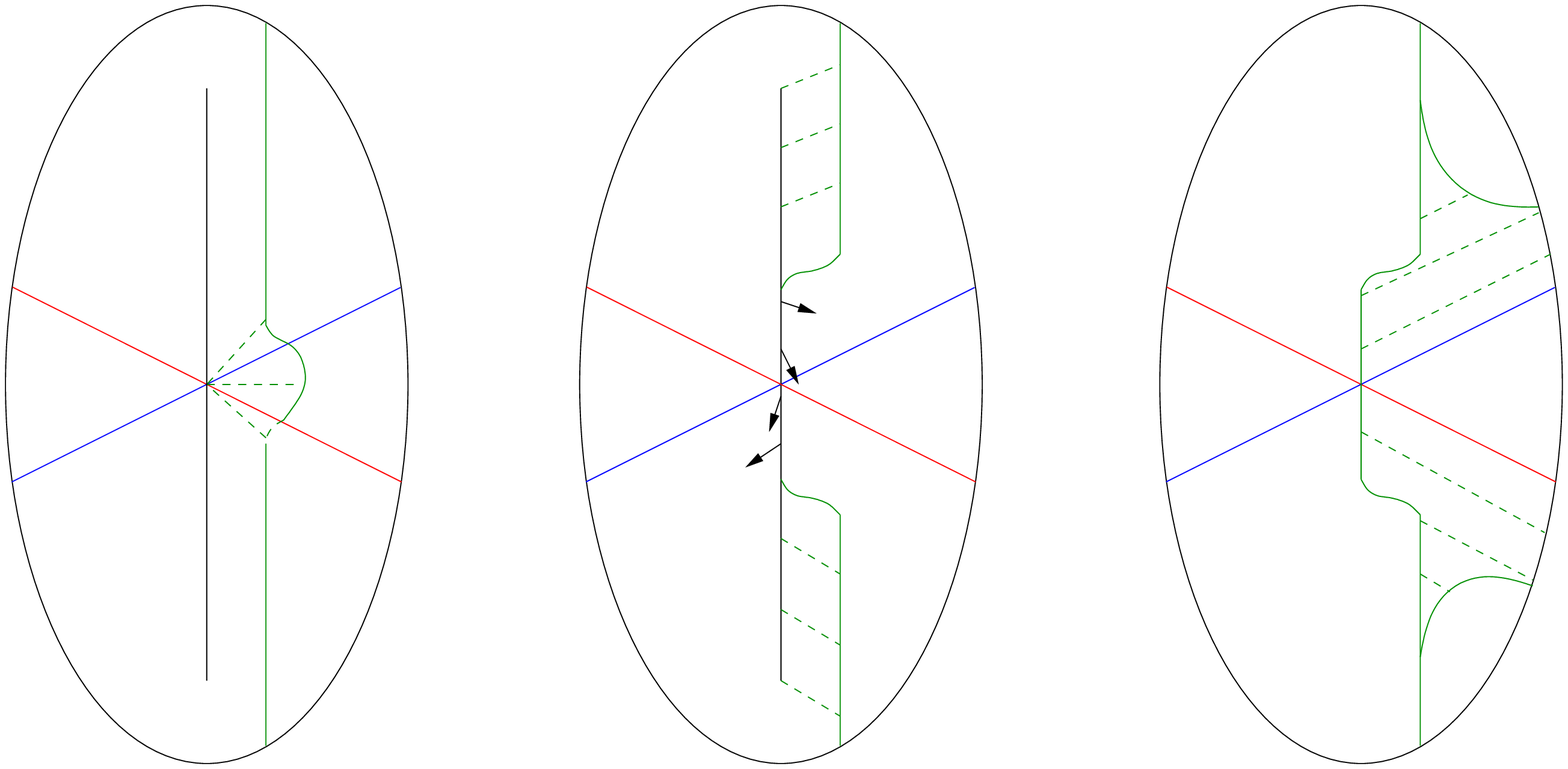}
    \put(8,10){$\gamma_{x_i}$}
    \put(17,10){$K'_{\xx}$}
    \put(21,30){\color{blue}$\beta$}
    \put(21,21){\color{red}$\alpha$}
    \put(45,10){$\gamma_{x_i}$}
    \put(54,10){$K_{\xx}$}
    \put(58,30){\color{blue}$\beta$}
    \put(58,21){\color{red}$\alpha$}
    \put(82,18){$K_{\xx}$}
    \put(92,8){$K_{\yy}$}
    \put(92,39){$K_{\yy}$}
    \put(11,-6){(a)}
    \put(47,-6){(b)}
    \put(85,-6){(c)}
    \end{overpic}
    \vspace{0.5cm}
    \caption{}
    \label{cobord}
\end{figure}\end{center}

If we are given a link cobordism between two links and a framing of one, then it induces a framing of the other. So $\tau_{\xx}$ induces a framing $\tau_{\yy}$ of $K_{\yy}$. The Pontryagin-Thom construction tells us that $\widetilde{\text{gr}}(\xx)-\widetilde{\text{gr}}(\yy)$ equals $\tau_{\yy}-\eta_{\yy}$. We will now compute this difference. Since $K_{\xx}$ and $K_{\yy}$ coincide as framed links outside of $\tilde{C}$, we only need to do this calculation in a neighborhood of $\tilde{C}$. To do so, we take a normal vector field $N$ to $\tilde{C}$ and extend it arbitrarily to $K_{\xx}\cap K_{\yy}$. So $N$ gives rise to a framing of $S$, which we call $\nu$. We denote by $\nu_{\xx}$ and $\nu_{\yy}$ the restrictions of $\nu$ to $K_{\xx}$ and $K_{\yy}$, resp. We will compute the difference between the framings by first comparing them with $\nu$ and then using the fact that
$$\tau_{\yy}-\eta_{\yy}=(\tau_{\yy}-\nu_{\yy})-(\eta_{\yy}-\nu_{\yy})=(\tau_{\xx}-\nu_{\xx})-(\eta_{\yy}-\nu_{\yy}).$$

\begin{center}\begin{figure}[ht]
    \begin{overpic}[scale=0.5]{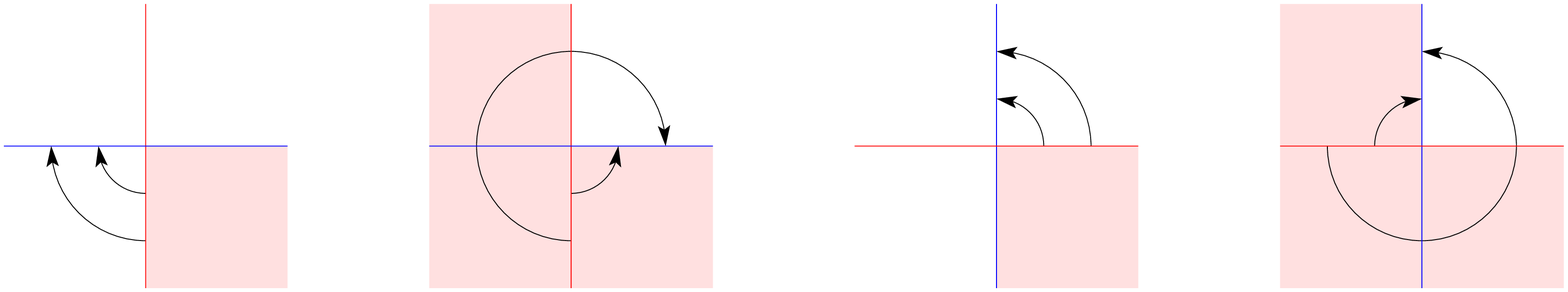}
    \end{overpic}
    \put(-442,25){$\tau_{\xx}$}
    \put(-461,25){$\nu_{\xx}$}
    \put(-287,25){$\tau_{\xx}$}
    \put(-320,25){$\nu_{\xx}$}
    \put(-160,55){$\eta_{\yy}$}
    \put(-145,55){$\nu_{\yy}$}
    \put(-64,55){$\eta_{\yy}$}
    \put(-30,55){$\nu_{\yy}$}
    \put(-445,-13){convex $x_i$}
    \put(-325,-13){concave $x_i$}
    \put(-185,-13){convex $y_j$}
    \put(-65,-13){concave $y_j$}
    \put(-395,33){\color{blue}$\beta$}
    \put(-425,5){\color{red}$\alpha$}
    \put(-267,33){\color{blue}$\beta$}
    \put(-297,5){\color{red}$\alpha$}
    \put(-138,33){\color{red}$\alpha$}
    \put(-168,5){\color{blue}$\beta$}
    \put(-10,33){\color{red}$\alpha$}
    \put(-40,5){\color{blue}$\beta$}
    \caption{}
    \label{diff}
\end{figure}\end{center}

We will look at a neighborhood of the corners of $F$. In fact we only need to compute how many times $\tau_{\xx}$ rotates with respect to $\nu_{\xx}$, where $K_{\xx}$ coincides with each $\gamma_{x_i}$ and similarly for $\eta_{\yy}$.
We call a nondegenerate corner of $F$ {\em convex}\footnote{Some authors use the adjectives acute and obtuse to denote convex and concave, respectively.} if it is a corner of some $D_k\subset F$ for only one $k$ and {\em concave}\footnotemark[1] if it is a corner of some $D_k\subset F$ for three values of $k$.
For convex vertices, the difference is 0 for both an $x_i$ and a $y_j$. For concave vertices, it is $+1$ for an $x_i$ and $-1$ for a $y_j$, as shown in Figure \ref{diff}. In this picture, the orientation of the link is pointing down, so a counterclockwise turn counts as a $+1$, since that is a left-handed twist. At the distinguished corner, we rotated $E_1$ by an additional $2\pi\chi(F)$ clockwise. If this is an $x_i$ it accounts for $\chi(F)$ in $\tau_{\xx}-\nu_{\xx}$ and if it is a $y_j$, it accounts for $-\chi(F)$ in $\eta_{\yy}-\nu_{\yy}$. So $\tau_{\yy}-\eta_{\yy}=\chi(F)+q$, where $q$ is the number of concave corners.

Now if we denote by $p$ the number of convex corners, by Lipshitz's formula,
\begin{eqnarray*}\text{ind}(F)&=&e(F)+n_{\xx}(F)+n_{\yy}(F)\\&=&\chi(F)-\tfrac{1}{4}p+\tfrac{1}{4}q+\tfrac{1}{4}p+\tfrac{3}{4}q\\&=&\chi(F)+q=\tau_{\yy}-\eta_{\yy}.\end{eqnarray*}
Since $n_{z}(F)=0$, we conclude that $\widetilde{\text{gr}}(\xx)-\widetilde{\text{gr}}(\yy)=\tau_{\yy}-\eta_{\yy}=\mu(A)=\text{gr}(\xx,\yy)$.

\textit{Step 2}: We will now prove a technical lemma that will be useful in the general case.

Given two links $K_1$ and $K_2$ in $Y$ that belong to the same homology class, let $S$ be an immersed cobordism between them. That means that $S$ is an immersed oriented compact surface in $Y\times [0,1]$ that is embedded near its boundary and such that $\partial S=K_1\times\{1\}\cup(-K_2)\times\{0\}$. Since an immersed surface also has a normal bundle, we can ask whether framings of $K_1$ and $K_2$ extend to a framing of $S$. So given a framing of $K_1$, the surface $S$ induces a framing of $K_2$. The induced framing of $K_2$ depends heavily on $S$. In fact, if we denote the signed number of self-intersections of $S$ by $\delta(S)$, we have the following lemma.
Here we orient $Y\times [0,1]$ by declaring that $\{\partial_t,E_1,E_2,E_3\}$ is an oriented basis, where $\{E_1,E_2,E_3\}$ is an oriented basis for $TY$ and $t$ is the coordinate function on $[0,1]$.

\begin{lemma}\label{l1}
 Let $K_1$ and $K_2$ be links in
$Y$ that belong to the same homology class and let $S$ and $S'$ be immersed cobordisms between them, which are in the same relative homology class. Given a framing of $K_1$, let $\zeta_S$ and $\zeta_{S'}$ be the framings induced on $K_2$ by $S$ and $S'$, respectively. Then $\zeta_{S}-\zeta_{S'}=2(\delta(S)-\delta(S'))$.
\end{lemma}

To prove that, we will use another lemma, which is a standard result in Differential Topology.

\begin{lemma}\label{l2}
 Let $\Sigma$ be a closed oriented surface immersed into a closed oriented 4-manifold $X$. Let $e(N_{\Sigma})$ be the Euler class ot the normal bundle of $\Sigma$ with the orientation induced by the orientation of $X$. Then
$$[\Sigma]\cdot [\Sigma]=e(N_{\Sigma})+2\delta(\Sigma).$$
\end{lemma}

\begin{proof}[Proof of Lemma \ref{l1}]
 We are given $S,S'\subset Y\times [0,1]$ such that $\partial S'=\partial S=K_1\times\{1\}\cup(-K_2\times\{0\})$ and such that $S'-S$ vanishes in $H_2(Y\times[0,1])$.
Now we take two copies of $Y\times [0,1]$, switch the orientation of one of them and glue along their common boundaries. We can think of this as $Y\times[-1,1]$ with the obvious identification of $Y\times\{-1\}$ and $Y\times\{1\}$, which gives us $Y\times S^1$. We can also glue $S\subset Y\times [0,1]$ to $-S'\subset Y\times[-1,0]$ and we get a closed surface that we call $\Sigma$. Now we can assume that in $Y\times[-\varepsilon,\varepsilon]$, the surface $\Sigma$ is $K_2\times[-\varepsilon,\varepsilon]$, for $\varepsilon$ small. We use $S$ to get a framing on $K_2\subset Y\times\{\varepsilon\}$ and $S'$ to get a framing on $K_2\subset Y\times\{-\varepsilon\}$. These are exactly $\zeta_S$ and $\zeta_{S'}$, respectively. It follows that the relative Euler class of the normal bundle of $\Sigma$ restricted to $K_2\times[-\varepsilon,\varepsilon]$ given these two framings is $\zeta_{S'}-\zeta_{S}$. Therefore $e(N_{\Sigma})=\zeta_{S'}-\zeta_{S}$.
Now, if we think of $S$, $S'$ and $\Sigma$ as chains in $Y\times S^1$, we can write $\Sigma=S-S'$. So $\Sigma-(K_1\times S^1)$ vanishes in $H_2(Y\times S^1)$. Hence $$[\Sigma]\cdot[\Sigma]=[K_1\times S^1]\cdot[K_1\times S^1]=0.$$ Therefore, by Lemma~\ref{l2},
$$\zeta_{S}-\zeta_{S'}=2\delta(\Sigma)=2(\delta(S)-\delta(S')).$$
\end{proof}

%\begin{proof}[Proof of Lemma \ref{l2}]
%We choose a section $\phi$ of $N_{\Sigma}$ whose zeroes are away from the self-intersections of $\Sigma$. Then we exponentiate $t\phi$ for some small $t$ to get an immersed surface $\Sigma'$. First we see that the zeroes of $\phi$ give rise to intersections of $\Sigma$ and $\Sigma'$ with the same sign. Then given a self-intersection point of $\Sigma$, a small neighborhood of it in $X$ consists of two embedded pieces $D_1$ and $D_2$ intersecting only at this point. Assuming that $t$ is small enough, $\text{exp}_{D_1}(t\phi)$ will intersect $D_2$ at a very close but different point and $\text{exp}_{D_2}(t\phi)$ will intersect $D_1$ also at a different and unique point. It is also clear that the sign of these two intersections are equal to the sign of the original self-intersection point. Moreover there are no other intersections of $\Sigma$ and $\Sigma'$. Therefore $$[\Sigma]\cdot[\Sigma]=\Sigma\cdot \Sigma'=e(N_{\Sigma})+2\delta(\Sigma).$$
%\end{proof}

\textit{Step 3}: We now proceed to the general case.
We had written $D(\varphi)$ as a union of surfaces $F_l\subset\Sigma$, which can be seen as 2-chains in $\Sigma$. We need to show that $$\widetilde{\text{gr}}(\xx)-\widetilde{\text{gr}}(\yy)= \sum_{l=1}^m\Big( e(F_l)+n_{\xx}(F_l)+n_{\yy}(F_l)-2n_z(F_l)\Big).$$

Let $\gamma_a$ be the projection to $\Sigma$ of the image of $\partial D^2\cap\{z;\text{Re}(z)\le 0\}$ under $\varphi$ and $\gamma_b$ be the projection of the image of $\partial D^2\cap\{z;\text{Re}(z)\ge 0\}$. Then $\gamma_a-\gamma_b=\partial D(A)=\sum_l\partial F_l$.
We observe that the a corner of $F_l$ can either be an $x_i$, a $y_j$ or neither. If it is neither of the two, then the interiors of $\gamma_a$ and $\gamma_b$ intersect at that point. We call this point an auxiliary corner and denote each of them by $w_k$ for some $k$. Now fix and auxiliary corner $w_k$. Let $r$ be the multiplicity of $\gamma_a$ and $s$ be the multiplicity of $\gamma_b$ in a neighborhood of $w_k$ and assume $r<s$, see Figure \ref{aux}(a). We might also have an extra $t$ to the multiplicity of all the four regions. But that will not affect the calculations. So, for simplicity, we can assume that $t=0$. We get a convex corner for $r$ of the $F_l$'s and a concave one for $r$ of the $F_l$'s. For $(s-r)$ of the $F_l$'s, this point lies on the boundary and is not a corner. We denote by $\gamma_{w_k}$ the flow line passing through $w_k$. We say that $w_k$ is positive if it behaves as a convex $x_i$ (i.e $\gamma_{w_k}$ is positively oriented) and as a concave $y_j$ (i.e $\gamma_{w_k}$ is negatively oriented), and that $w_k$ is negative if the opposite happens, as shown in Figure \ref{aux}(b).

The orientations on $\gamma_a$ and $-\gamma_b$ give rise to an orientation of $\partial F_l$. That is also the orientation induced from $\Sigma$, since $A\ge 0$. Now we need to define $\{E_1,E_2,E_3\}$. We want to define $E_1$ on $F_l$ in the same way as we did when we had only one $F_l$. But we have to be more careful since we may have $\alpha$ and $\beta$ curves contained on the surface $F_l$. This can happen in three different ways: there is a boundary degenerate corner, an interior degenerate corner or a pair of nondegenerate corners that are on $\partial F_l$ but are not corners of $\partial F_l$ for some $l$. Figure \ref{degn} shows an example of each of those case.

\begin{center}\begin{figure}[ht]
    \begin{overpic}[scale=0.5]{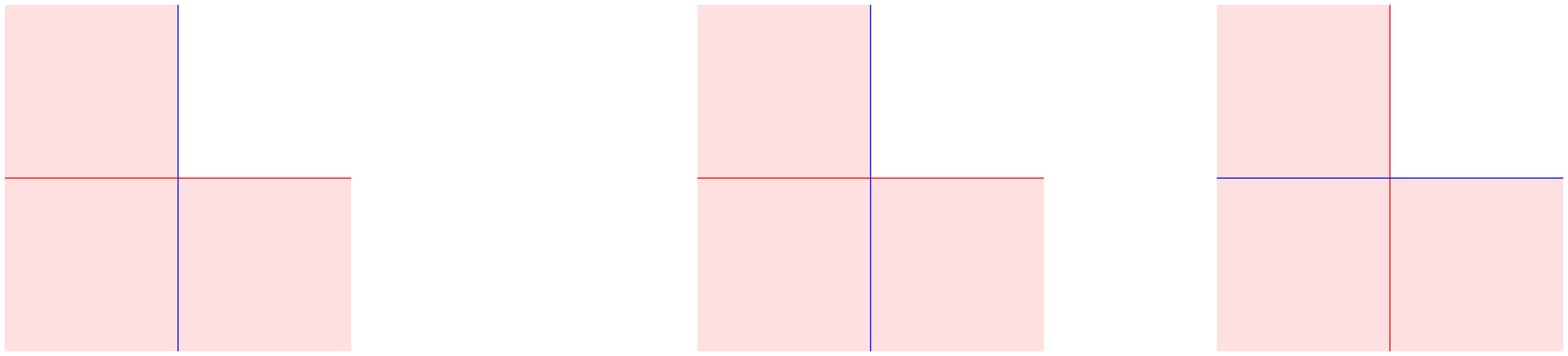}
    \put(5,16){$s$}
    \put(16,5){$r$}
    \put(2.5,5){$r+s$}
    \put(50,-3){positive $w_k$}
    \put(82,-3){negative $w_k$}
    \put(12.5,19){\color{blue}$\beta$}
    \put(57,19){\color{blue}$\beta$}
    \put(90,19){\color{red}$\alpha$}
    \put(20,12){\color{red}$\alpha$}
    \put(65,12){\color{red}$\alpha$}
    \put(98,12){\color{blue}$\beta$}
    \put(10,-7){(a)}
    \put(70,-7){(b)}
    \end{overpic}
    \vspace{0.5cm}
    \caption{}
    \label{aux}
\end{figure}\end{center}

\begin{center}
\begin{figure}[ht]
 \begin{overpic}[scale=0.5]{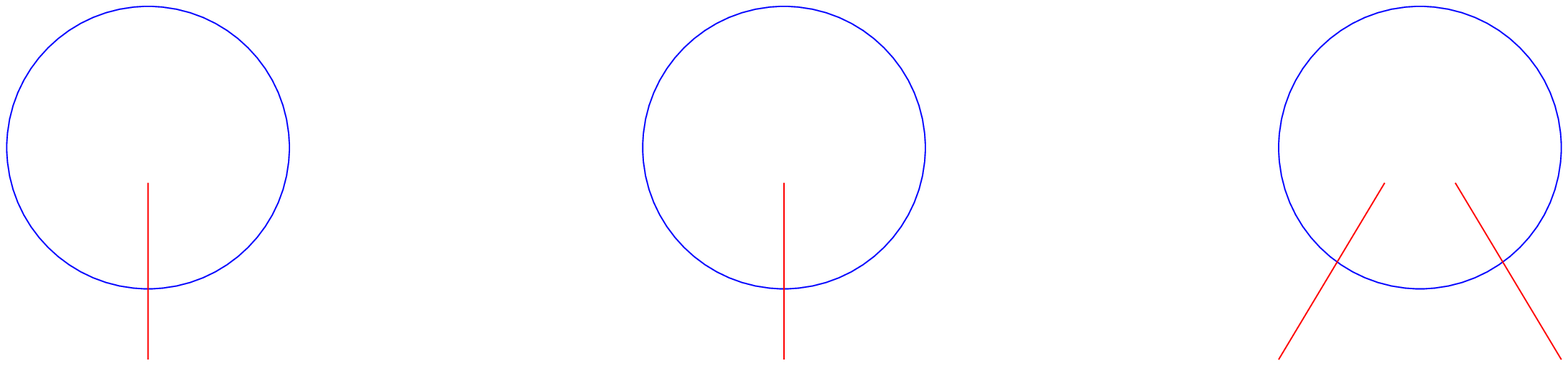}
  \put(6,7){0}
  \put(11,7){0}
  \put(6,1){1}
  \put(11,1){1}
  \put(47,7){1}
  \put(52,7){1}
  \put(47,1){1}
  \put(52,1){1}
  \put(90,7){0}
  \put(84.5,9){0}
  \put(95.5,9){0}
  \put(90,1){2}
  \put(82,5){1}
  \put(98,5){1}
  \put(-6,-7){boundary degenerate}
  \put(35,-7){interior degenerate}
  \put(74,-7){pair of nondegenerate}
 \end{overpic}
  \vspace{0.5cm}
  \caption{}
\label{degn}
\end{figure}
\end{center}

For each $F_l$, we can define $C_l$, just as we did to define $C$ in Step 1, except that when one of the edges of $F_l$ is a circle, we will attach a disk to it, not a triangular surface.
We will first define $E_1$ on $F_m$. For each edge of $F_m$ that is not a circle, we define $E_1$ to be the positive unit tangent vector to $\partial F_m$ outside neighborhoods of the corners. Along an edge that is a circle, we define $E_1$ to be any vector field whose rotation number along this circle is 0. We note that nondegenerate corners along this circle, e.g. Figure \ref{degn}, cannot happen for $F_m$. If we have an $\alpha$ or $\beta$ circle contained in the interior of $F_m$, then we define $E_1$ along this circle such that its rotation number is 0. In a neighborhood of each corner including the auxiliary ones, we rotate $E_1$ as least as possible, as we did in Step 1. We also need to choose some nondegenerate corners, i.e. not auxiliary corners, to rotate a total of $\chi(F_m)+d(F_m)$, where $d(F_m)$ denotes the number of boundary degenerate corners of $F_m$. After doing that, we can now extend $E_1$ to a vector field on $F_m$. Now we extend it to the triangular surfaces belonging to $C_m$ just as we did in Step 1. For each circle on $\partial F_m$, we extend $E_1$ to the attaching disk by requiring that it is tangent to the surface $f^{-1}(t)$, for every $3/2\le t\le2$, if the circle is a $\beta_j$ and for every $1\le t\le 3/2$ if the circle is an $\alpha_i$. We note that $E_1$ is not tangent to this disk at any point except for the corresponding critical point, i.e when $t=1$ or $2$, and on $\Sigma$.

Now we want to extend $E_1$ to $F_{m-1}\supset F_m$. We first define $E_1$ on $\partial F_{m-1}$. We can do it the same way as we did for $\partial F_m$ except near the intersection of  $\partial F_{m-1}$ and $F_m$, where $E_1$ is already defined. This can only happen in two cases. The first one is when they intersect at an auxiliary corner. In this case we just rotate $E_1$ along $\partial F_{m-1}$ as least as possible, so that it coincides with $E_1$ at the corner. The second case is when there is a circle in $F_{m-1}$ that contains two nondegenerate corners. In this case, $E_1$ is already defined in the segment connecting the two nondegenerate corners. So we extend it to all of this circle in such a way that its rotation number is 0. After doing that, we can extend $E_1$ to $C_{m-1}$ just as we did for $C_m$. Proceeding by induction, we define $E_1$ on $C_l$, for $l=m,m-1,\dots,1$.

We can define $E_3$ on $C_l$ as we did before, but when we have a circle on $\partial C_l$, we extend $E_3$ to the corresponding disk by requiring that $E_3$ is normal to $f^{-1}(t)$ for every $t$. Now we define $E_2$ such that $\{E_1,E_2,E_3\}$ is an orthonormal basis for $TY$ along $C_l$ for all $l$.

For every $\alpha$ or $\beta$ circle contained in $F_1$, either we have attached the corresponding disk to it in some $C_l$ or it contains an interior degenerate corner, in which case, we have also required that the rotation number of $E_1$ along this circle is 0. So in the latter case, we can extend $E_1$ and $E_3$ as we did when the circle was in the boundary. Now, there is no obstruction to extending the orthonormal frame $\{E_1,E_2,E_3\}$ to all of $Y$ and, as before, that determines a trivialization by sending $E_i$ to $e_i\in \R^3$.

Again, we take $K_{\xx}'=w_{\xx}^{-1}(e_1)$ and $K_{\yy}'=w_{\yy}^{-1}(e_1)$. We can isotope them the same way as before to get $K_{\xx}$ and $K_{\yy}$ so that they contain segments of $\gamma_{x_i}$ and $\gamma_{y_i}$ near the respective corners. We also define the surfaces $\tilde{C}_l$ in the same fashion as we did in Step 1. Now, to compute the difference of their framings, we will use several immersed cobordisms. We start from $K_{\yy}$. We use $\tilde{C}_1$ to define an immersed cobordism. This cobordism exchanges segments of the flow lines $\gamma_{y_j}$ corresponding to corners $y_j$ of $F_1$ with segments of some $\gamma_{x_i}$ corresponding to corners $x_i$ of $F_1$ and possibly segments of some $\gamma_{w_k}$, corresponding to concave auxiliary corners $w_k$. The next step is to use $\tilde{C}_2$ to construct an immersed cobordism which exchanges segments of some $\gamma_{y_i}$ by segments of some $\gamma_{x_i}$, possibly involves auxiliary corners and keeps the rest of the link fixed. We can continue this construction inductively and define immersed cobordisms for $\tilde{C}_1,\dots,\tilde{C}_m$. Every time we obtain a $\gamma_{w_k}$, it will first appear as a concave corner and later as a convex corner. If $w_k$ is positively oriented, then it will appear as a positive concave angle and a negative convex angle, which means that they just cancel, when we stack the immersed cobordisms. If $w_k$ is negatively oriented, then it will appear as a negative concave corner first and as a positive convex corner later. In this case, we add trivial cobordisms to the immersed cobordisms where the segment of $\gamma_{w_k}$ appears and to all of the ones in between. After stacking all those, the auxiliary corners cancel and we obtain an immersed cobordism from $K_{\yy}$ to $K_{\xx}$. Similarly to the case when we had only one $F_l$, we conclude that the difference of the framings using the cobordism induced by $\tilde{C}_l$ is $\chi(F_l)+d(F_l)+q(F_l)$ for each $l$, where $q(F_l)$ is the number of concave corners of $F_l$, not counting the auxiliary corners. Moreover for each auxiliary corner $w_k$, the difference of framings is $+1$ if $w_k$ is positive, and $-1$ if $w_k$ is negative. So using this immersed cobordism from $K_{\yy}$ to $K_{\xx}$, the difference between the framings is
$\sum_{l=1}^m \Big(\chi(F_l)+d(F_l)+q(F_l)\Big)$ plus the signed count of the auxiliary corners.

We know that there is an embedded link cobordism from $K_{\yy}$ to $K_{\xx}$ in the same relative homology class as the immersed cobordism we were considering. So, by Lemma~\ref{l1}, $\tau_{\yy}-\eta_{\yy}$ equals the difference obtained using the immersed cobordism minus twice the signed number of self-intersections of the immersed cobordism, since the self-intersection number of an embedded cobordism is 0. We now need to consider three cases.
\begin{itemize}
 \item[(i)] There are boundary degenerate corners or a pair of nondegenerate corners on an $\alpha$ or $\beta$ curve contained in some $\partial F_l$.
\item[(ii)] There are interior degenerate corners
\item[(iii)] There are nondegenerate corners in the interior of some $F_l$.
\item[(iii)] The basepoint $z$ in in the interior of $F_1$.
\end{itemize}

In case (i), self-intersections could exist if $K_{\xx}$ or $K_{\yy}$ intersects $C_l$ for $l$ such that $C_l$ contains the disk we attach to the corresponding $\alpha$ or $\beta$ circle. Let $x_i$ and $y_j$ be the corresponding corners. Then $C_l$ divides $N(\gamma_{x_i})$ in two disconnected components and we can see that $K_{\xx}$ enters and exits $N(\gamma_{x_i})$ in the same component. Similarly for $y_j$. Therefore the signed number of intersections with $C_l$ is 0. In this case, $n_{x_i}+n_{y_j}=1$. But this $+1$ appears in the difference of framings when we added $d(F_l)$ turns to $E_1$ near a nondegenerate corner.

In case (ii), let $x_i=y_j$ be the interior degenerate corner. So, $n_{x_i}+n_{y_j}=2$. Also, $K_{\xx}=K_{\yy}$ in $N(\gamma_{x_i})$. Also, $K_{\xx}$ intersects $C_l$ negatively at only one point. Therefore, by Lemma \ref{l1}, we have two add +2 to the difference of the framings.

In case (iii), since $F_i\supset F_j$, for $i<j$, and the cobordism corresponding to $\tilde{C}_i$ is taken before the one corresponding to $\tilde{C}_j$, only the nondegenerate $y_j$'s which are in the interior of an $F_j$ correspond to intersections. So, by Lemma \ref{l1}, we have to add twice the number of interior nondegenerate $y_j$'s. On the other hand, if we had built our immersed cobordisms in the opposite order, i.e. starting with $F_m$ and going all the way to $F_1$, then we would get the same result, except that we would be counting twice the number of interior nondegenerate corners $x_i$, but in this case the sign of the auxiliary corners are switched. Since the two calculations have to coincide, it follows that the number of interior nondegenerate corners $x_i$ plus the number of positive auxiliary corners equals the number of interior nondegenerate corners $y_j$ plus the number of negative auxiliary corners. So twice the number of interior nondegenerate $x_i$'s plus the signed count of the auxiliary corners equals the total number of interior nondegenerate corners. That is exactly what we were missing to get the full $n_{\xx}(F_l)$ and $n_{\yy}(F_l)$. Therefore, combining cases (i),(ii) and (iii), we conclude that the difference of the framings is $\sum_{l=1}^m \Big(e(F_l)+n_{\xx}(F_l)+n_{\yy}(F_l)\Big)$, which is equal to $\mu(A)$.

In case (iv), then $K_{\xx}=K_{\yy}$ near $\gamma_z$. If $K_{\xx}$ intersects $F_l$, then it does so positively. Hence, by Lemma \ref{l1}, we get an extra $-2\sum_l n_z(F_l)$ in the difference of framings.
Therefore $$\widetilde{\text{gr}}(\xx)-\widetilde{\text{gr}}(\yy)=\tau_{\yy}-\eta_{\yy}=\mu(A)-2n_z(A)=\text{gr}(\xx,\yy).$$

\section{The absolute grading of the contact invariant} \label{SecOnContCls}

In ~\cite{OSz1}, Oszv\'ath-Szab\'o defined the contact class $c(\xi) \in \widehat{HF}(-Y)$ for a contact 3-manifold $(Y,\xi)$, and they showed that it is an invariant of $\xi$. Later, Honda-Kazez-Mati\'c~\cite{HKM} gave an alternative definition of $c(\xi)$ using an open book decomposition adapted to $\xi$. In this section, we compute the absolute grading of the contact invariant $c(\xi)$.

\subsection{Contact topology and open book decompositions}

Let $Y$ be a closed oriented 3-manifold. A contact structure $\xi$ is a maximally non-integrable co-oriented 2-plane field, i.e. there exists a 1-form $\lambda$ such that $\lambda \wedge d\lambda>0$ and $\xi=ker\lambda$. We call such $\lambda$ a {\em contact form} of $\xi$. The {\em Reeb vector field} $R_\lambda$ associated with $\lambda$ is the unique vector field which satisfies (i) $R_\lambda ~\lrcorner~ d\lambda=0$, (ii) $R_\lambda ~\lrcorner~ \lambda=1$. Although the dynamics of $R_\lambda$ depend heavily on the choice of $\lambda$, its homotopy class is an invariant of $\xi$. In fact, two contact structures are homotopic if and only if their associated Reeb vector fields are homotopic.

Now recall that an {\em open book decomposition} of $Y$ is a pair $(S,h)$, where $S$ is a compact, oriented surface of genus $g$ with boundary, $h:S \to S$ is a diffeomorphism which is the identity on $\bdry S$, and $Y$ is homeomorphic to $(S\times[0,1])/\sim$. The equivalence relation $\sim$ is defined by $(x,1) \sim (h(x),0)$ for $x \in S$ and $(y,t) \sim (y,t')$ for $y \in \bdry S$ and $t,t' \in [0,1]$. Given a contact structure $\xi$ on $Y$, an open book $(S,h)$ is {\em adapted to} $\xi$ if there exists a contact form $\lambda$ for $\xi$ such that $R_\lambda$ is positively transverse to $int(S)$ and positively tangent to $\bdry S$.

Fix an adapted open book $(S,h)$ of $(Y,\lambda)$. Following~\cite{HKM}, let $\{a_1,\cdots,a_{2g}\}$ be a set of pairwise disjoint, properly embedded arcs on $S$ such that $S \setminus \bigcup_{i=1}^{2g} a_i$ is a single polygon. We call $\{a_1,\cdots,a_{2g}\}$ a {\em basis} for $S$. Next let $b_i$ be an arc which is isotopic to $a_i$ by a small isotopy so that the following hold:

\be
\item{The endpoints of $a_i$ are isotoped along $\bdry S$, in the direction given by the boundary orientation of S.}
\item{$a_i$ and $b_i$ intersect transversely in one point $x_i$ in the interior of S.}
\item{If we orient $a_i$, and $b_i$ is given the induced orientation from the isotopy, then the sign of the intersection $a_i \cap b_i$ is $+1$.}
\ee

See Figure~\ref{BasisArcs}.\\

\begin{figure}[ht]
    \begin{overpic}[scale=.3]{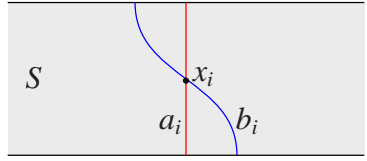}
    \put(42.5,7.5){$a_i$}
    \put(64,7.5){$b_i$}
    \put(51.5,21){$x_i$}
    \put(5,19){$S$}
    \end{overpic}
    \caption{The arcs $a_i$ and $b_i$ on $S$.}
    \label{BasisArcs}
\end{figure}

Observe that $(S,h)$ naturally induces a Heegaard splitting of $Y$ by letting $H_1=(S\times[0,1/2])/\sim$ and $H_2=(S\times[1/2,1])/\sim$. This gives a Heegaard decomposition of $Y$ of genus $2g$ with Heegaard surface $\Sigma=\bdry H_1=-\bdry H_2$. By choosing a basis $\{a_1,\cdots,a_{2g}\}$ for $S$ and following the constructions above, we obtain two collections of simple closed curves $\boldsymbol\alpha=\{\alpha_1,\cdots,\alpha_{2g}\}$ and $\boldsymbol\beta=\{\beta_1,\cdots,\beta_{2g}\}$ on $\Sigma$, where $\alpha_i=\bdry (a_i\times[0,1/2])$ and $\beta_i=\bdry (b_i\times[1/2,1])$ for $i=1,\cdots,2g$. Then one can properly place the basepoint $z$ and reverse the orientation of $Y$ to obtain a weakly admissible Heegaard diagram $(\Sigma,\boldsymbol\beta,\boldsymbol\alpha,z)$ for $-Y$. It is observed in~\cite{HKM} that $\mathbf{x}=(x_1,\cdots,x_{2g}) \in \widehat{CF}(\Sigma,\bm\beta,\bm\alpha,z)$ defines a cycle, where $x_i=a_i \cap b_i \in \alpha_i \cap \beta_i$, $i=1,\cdots,2g$.

\begin{thm}[Honda-Kazez-Mati\'c \cite{HKM}] \label{HKM}
The class $[\mathbf{x}] \in \widehat{HF}(-Y)$ represented by $\mathbf{x} \in \widehat{CF}(\Sigma,\bm\beta,\bm\alpha,z)$ from above is an invariant of $\xi$ and it is equal to $c(\xi)$ defined in \cite{OSz1}.
\end{thm}

\begin{rmk}
In light of Theorem~\ref{HKM}, in order to prove Theorem~\ref{mainThm}(b), it suffices to show
\begin{equation} \label{ContClsEq}
    \widetilde{\text{gr}}(\mathbf{x})=[\xi]
\end{equation}
as homotopy classes of oriented 2-plane fields.
\end{rmk}

\subsection{Proof of Theorem~\ref{mainThm}(b)}

Throughout this section, we fix a contact form $\lambda$ and an adapted open book decomposition $(S,h)$ of $(Y,\lambda)$. Note that the contact invariant is presented as an intersection point $\mathbf{x}$ in $\widehat{CF}(-Y)$. The plan is to use the Pontryagin-Thom construction to show that the vector field constructed in Section~\ref{SecDef} to define $\widetilde{\text{gr}}(\mathbf{x})$ is homotopic to the Reeb vector field $R_\lambda$.

\begin{proof}[Proof of Theorem~\ref{mainThm}(b)]
Let $f$ be a Morse function adapted to our special Heegaard diagram $(\Sigma,\bdsym\alpha,\bdsym\beta,z)$, where $\Sigma=(S\times\{0\}) \cup (S\times\{1/2\})$. Note that one needs to reverse the orientation of $Y$ to define $[\mathbf{x}]=c(\xi)$. Equivalently, we shall consider, for the rest of the proof, the same Heegaard diagram $(\Sigma,\bdsym\alpha,\bdsym\beta,z)$, but with the downward gradient vector field $-\nabla f$. All the constructions of the absolute grading function carry over by simply reversing the direction of all vector fields.
Let $v_{\xx}$ be a nonvanishing vector field, which is a modification of $-\nabla f$, as defined in Section~\ref{SecDef}. In particular, the homotopy class of the orthogonal complement of $v_{\xx}$ equals $\widetilde{\text{gr}}(\xx)$.
Let $\tilde S \subset int(S)$ be a closed subsurface such that $S$ deformation retracts onto $\tilde S$, and assume that $h$ is supported in $\tilde S\times\{1\}$. It is easy to see that $-\nabla f$ is homotopic to $R_\lambda$ by linear interpolation in a small neighborhood $N(\tilde S\times\{1\})$ of $\tilde S\times\{1\}$ in $M$ because they are both positively transverse to $\tilde S\times\{1\}$. Let $H=Y \setminus N(\tilde S\times\{1\})$ be the genus $2g$ handlebody\footnote{In fact $H$ is a handlebody with corners, but this is irrelevant here because we are considering continuous vector fields.}. So it suffices to show that $v_{\xx}|_H$ is homotopic to $R_\lambda|_H$ relative to $\bdry H$.

To do so, consider a closed collar neighborhood $a_i \times [-1,1] \subset S\times\{1/2\}$ of $a_i$ on the middle page such that it contains $b_i$ in the interior, for $i=1,\cdots,2g$. Let $B_i=(a_i \times [-1,1] \times [0,1]) \cap H \subset H$ be a 3-ball (with corners) in $H$, which contains $a_i$ and $b_i$ in the interior. See Figure~\ref{ReebandGrad} for pictures of the vector fields $R_\lambda|_{B_i}$ and $-\nabla f|_{B_i}$.\\

\begin{figure}[ht]
    \begin{overpic}[scale=.5]{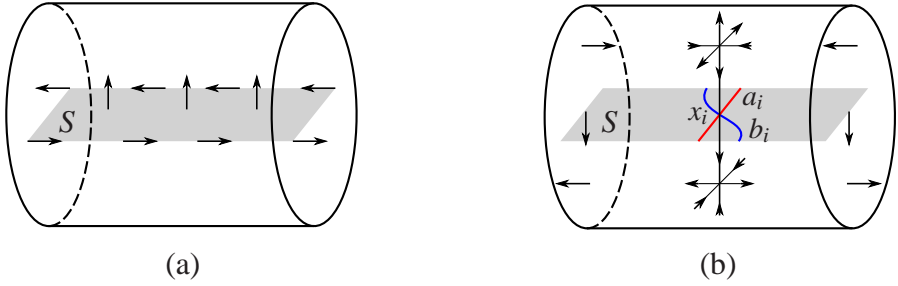}
    \put(82.8,14){\small{$a_i$}}
    \put(83.5,10.3){\small{$b_i$}}
    \put(76.6,12.3){\small{$x_i$}}
    \put(6,11){$S$}
    \put(67,11){$S$}
    \put(18,-5){(a)}
    \put(78,-5){(b)}
    \end{overpic}
    \newline
    \caption{(a) The Reeb vector field $R_\lambda$ restricted to $B_i$. (b) The downward gradient vector field $-\nabla f$ restricted to $B_i$.}
    \label{ReebandGrad}
\end{figure}

\n
{\em Claim:} There exists a non-singular vector field $R'_\lambda$ on $H$, homotopic to $R_\lambda$ relative to $\bdry H$, such that (i) $R'_\lambda|_{\bdry B_i}=v_{\xx}|_{\bdry B_i}$, (ii) $R'_\lambda|_{B_i}$ is homotopic to $v_{\xx}|_{B_i}$ relative to $\bdry B_i$, for $i=1,\cdots,2g$.

\begin{proof}[Proof of Claim]
Let $D_l=(a_i\times\{-1\}\times[0,1]) \cap H$ and $D_r=(a_i\times\{1\}\times[0,1]) \cap H$ be the left and right disk boundaries of $B_i$, respectively. Observe that $R_\lambda=v_{\xx}$ on $\bdry B_i \setminus (D_l \cup D_r)$ by construction. We shall consider a collar neighborhood $N(D_l)=(a_i\times[-1-\delta,-1+\delta]\times[0,1]) \cap H$ of $D_l$ for some small $\delta>0$, and homotope $R_\lambda$ to $R'_\lambda$ with the desired properties within $N(D_l)$. Note that the same construction can be carried over to a collar neighborhood of $D_r$.

We construct a model vector field $V_l$ on $D^2\times[-1,1]$ in steps. First let $\mathcal{F}_0$ be a singular foliation on $D^2$ which has two elliptic singularities as depicted in Figure~\ref{modReeb}(a). Let $\gamma \subset D^2\times[-1,0]$ be a properly embedded, boundary parallel arc such that $\bdry \gamma$ is exactly the union of the two singularities of $\mathcal{F}_0$ on $D^2\times\{-1\}$. Then there exists a foliation $\mathcal{F}$ by disks on $D^2\times[-1,0]$ such that for any leaf $F$ of $\mathcal{F}$, we have $\bdry F \cap int(D^2\times[-1,0])=\gamma$, and $\bdry F \cap (D^2\times\{-1\})$ is a leaf of $\mathcal{F}_0$. Let $V'_l$ be a non-singular vector field on $D^2\times[-1,0]$ such that it is positively tangent to $\gamma$ and positively transverse to the interior of all leaves of $\mathcal{F}$ as depicted in Figure~\ref{modReeb}(b). Up to homotopy, we can assume that $V'_l|_{D^2\times\{0\}}=v_{\xx}|_{D_l}$ as vector fields on a disk. By fixing a trivialization of the tangent bundle $T(D^2 \times [-1,1])$ using the standard embedding $D^2\times[-1,1] \subset \mathbb{R}^3$, we define the vector field $V_l$ on $D^2\times[-1,1]$ by
\begin{equation*}
    V_l(x,t)=
    \begin{cases}
    ~V'_l(x,t) & \text{if } -1 \leq t \leq 0,\\
    ~V'_l(x,-t) & \text{if } 0 \leq t \leq 1.
    \end{cases}
\end{equation*}
where $x \in D^2$ is any point. Identify $D^2\times[-1,1]$ with $N(D_l)$ by rescaling in the $[-1,1]$-direction such that $D_l$ is identified with $D^2\times\{0\}$, $N(D_l) \setminus B_i$ is identified with $D^2\times[-1,0]$, and $N(D_l) \cap B_i$ is identified with $D^2\times[0,1]$. It is easy to see that $R_\lambda|_{N(D_l)}$ is homotopic to $V_l$ as vector fields on $N(D_l)$ relative to the boundary. Similarly, one can define a non-singular vector field $V_r$ on $N(D_r)$ such that $R_\lambda|_{N(D_r)}$ is homotopic to $V_r$ as vector fields on $N(D_r)$ relative to the boundary. By applying the above homotopy, which is supported in $N(D_l) \cup N(D_r)$, to $R_\lambda$, and repeat this process for every $B_i$, $i=1,\cdots,2g$, we obtain a new non-singular vector field $R'_\lambda$. Observe that $R'_\lambda$ satisfies condition (i) by construction.

\begin{figure}[ht]
    \begin{overpic}[scale=.55]{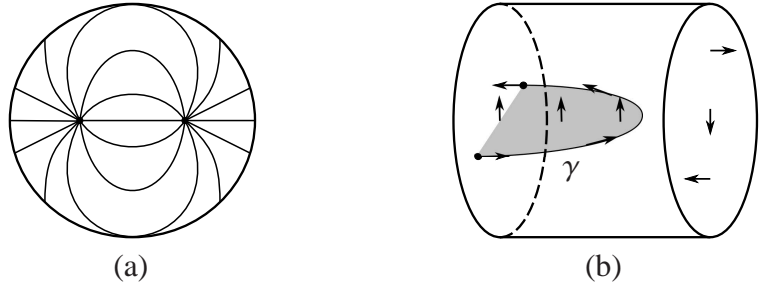}
    \put(14,-5){(a)}
    \put(77,-5){(b)}
    \put(74,8.5){$\gamma$}
    \end{overpic}
    \vspace{4mm}
    \caption{(a) The singular foliation on $D^2$. (b) The vector field $V_l'$ on a leaf of $\mathcal{F}$ in $D^2\times[-1,0]$.}
    \label{modReeb}
\end{figure}

To show that $R'_\lambda$ satisfies condition (ii), we use the Pontryagin-Thom construction. Trivialize the tangent bundle $TB_i$ by embedding $B_i \subset \mathbb{R}^3$ such that $D_l$ (or $D_r$) is parallel to the $xz$-plane, and the $[-1,1]$-direction is parallel to the $y$-axis. Consider the associated Gauss maps $G_{v_{\xx}}|_{B_i}: B_i \to S^2$ and $G_{R'_\lambda}|_{B_i}: B_i \to S^2$. Without loss of generality, we assume that $G_{v_{\xx}}|_{B_i}$ and $G_{R'_\lambda}|_{B_i}$ are smooth, and $p=(0,1,0) \in S^2$ is a common regular value. Let $p'=(\epsilon,\sqrt{1-\epsilon^2},0) \in S^2$ be a nearby common regular value which keeps track of the framing, where $\epsilon>0$ is small. It is now a straightforward computation that the Pontryagin submanifolds $G^{-1}_{v_{\xx}}(p)$ and $G^{-1}_{R'_\lambda}(p)$ are both framed cobordant to the framed arc depicted in Figure~\ref{PontThom} relative to the boundary. Hence $R'_\lambda|_{B_i}$ is homotopic to $v_{\xx}|_{B_i}$ relative to $\bdry B_i$, for all $i=1,\cdots,2g$. This finishes the proof of the claim.
\end{proof}

\begin{figure}[ht]
    \begin{overpic}[scale=.35]{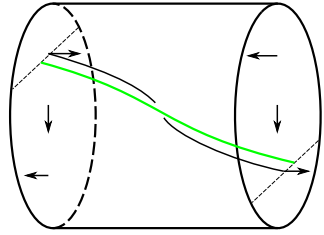}
    \end{overpic}
    \caption{A framed arc in $B_i$, where the framing is indicated by the green arc.}
    \label{PontThom}
\end{figure}

It remains to show that $R'_\lambda$ is homotopic to $v_{\xx}$ on $\overline{H \setminus (\bigcup_{i=1}^{2g}B_i)}$ relative to the boundary. Let $(D^2,\text{id})$ be the trivial open book of $S^3$, and $\tilde D \subset int(D^2)$ be a slightly smaller disk. Let $\tilde{H}$ denote $\overline{H \setminus (\bigcup_{i=1}^{2g}B_i)}$ and observe that it is naturally identified with $(D^2\times[0,1] \setminus ((\tilde D\times[0,\epsilon)) \cup (\tilde D\times(1-\epsilon,1])))/\sim$ by construction. On the one hand, it is easy to see that $R'_\lambda|_{\tilde{H}}$ is homotopic to the restriction of the Reeb vector field compatible with the open book $(D^2,\text{id})$. On the other hand, note that $\tilde{H}$ is nothing but a neighborhood of the gradient trajectory which connects the index 0 critical point to the index 3 critical point. Hence it follows immediately from our construction of $\widetilde{\text{gr}}(\mathbf{x})$ that $v_{\xx}|_{\tilde{H}}$ is also homotopic to the Reeb vector field compatible with $(D^2,\text{id})$. This finishes the proof of Theorem~\ref{mainThm}(b).
\end{proof}

Now we compute the twisted absolute grading of the twisted contact invariant defined in~\cite{OSz6}. Let $\mathbf{x}\in\mathbb{T}_\alpha\cap\mathbb{T}_\beta$ be the generator in $\widehat{CF}(-Y)$, which defines the usual contact invariant as before. Let $\mathbb{Z}[H^1(Y;\mathbb{Z})]^\times$ denote the set of invertible elements in $\mathbb{Z}[H^1(Y;\mathbb{Z})]$. First recall that the twisted contact invariant $\underline{c}(\xi)$ associated with the contact structure $\xi$ is defined by
\begin{equation*}
    \underline{c}(\xi)=[u\cdot\mathbf{x}] \in \widehat{\underline{HF}}(-Y)/\mathbb{Z}[H^1(Y;\mathbb{Z})]^\times
\end{equation*}
where $u\in\mathbb{Z}[H^1(Y;\mathbb{Z})]^\times$. Although $\underline{c}(\xi)$ is only well-defined up to a unit in $\mathbb{Z}[H^1(Y;\mathbb{Z})]$, the twisted absolute grading $\widetilde{gr}_{tw}(\underline{c}(\xi))$ defined by (\ref{twgr}) still makes sense. The following result is immediate.

\begin{cor} \label{twgrcontcls}
If $\xi$ is a contact structure on $Y$, then $\widetilde{gr}_{tw}(\underline{c}(\xi))=[\xi] \in \mathcal{P}(Y)$.
\end{cor}

\begin{proof}
This follows immediately from (\ref{twgr}) and Theorem~\ref{mainThm}(b).
\end{proof}

Now we are ready to prove the corollaries given in Section 1.

\begin{proof}[Proof of Corollary~\ref{KM97}]
If $(Y,\xi)$ is strongly fillable, then $c(\xi) \neq 0 \in \widehat{HF}(-Y)$ according to~\cite{OSz1}. Since $\widehat{HF}(-Y)$ is a finitely generated Abelian group, there can be only finitely many absolute gradings, i.e., homotopy classes of 2-plane fields, that support strongly fillable contact structures.

Now if $(Y,\xi)$ is weakly fillable, then $\underline{c}(\xi) \neq 0 \in \widehat{\underline{HF}}(-Y)/\mathbb{Z}[H^1(Y;\mathbb{Z})]^\times$ according to~\cite{OSz6}. Since $\widehat{\underline{HF}}(-Y)$ is finitely generated as a $\mathbb{Z}[H^1(Y;\mathbb{Z})]$ module, the same argument as above together with Corollary~\ref{twgrcontcls} implies that there can be only finitely many homotopy classes of 2-plane fields in $Y$ that support weakly fillable contact structures.
\end{proof}

\begin{proof}[Proof of Corollary~\ref{L-sp}]
By definition if $Y$ is an $L$-space, then $\widehat{HF}(-Y)$ is a free Abelian group of rank $|H_1(Y;\mathbb{Z})|$. Therefore there are at most $|H_1(Y;\mathbb{Z})|$-many homotopy classes of 2-plane fields that support strongly fillable contact structures. To get the same result for weakly fillable contact structures, it suffices to observe that since $Y$ is a rational homology sphere by assumption, we have
\begin{equation*}
    \widehat{\underline{HF}}(-Y) \simeq \widehat{HF}(-Y) \otimes \mathbb{Z}[H^1(Y;\mathbb{Z})].
\end{equation*}
Hence $\widehat{\underline{HF}}(-Y)$ is a free $\mathbb{Z}[H^1(Y;\mathbb{Z})]$ module of rank $|H_1(Y;\mathbb{Z})|$, and therefore the conclusion follows as before.
\end{proof}

\begin{proof}[Proof of Corollary~\ref{Lis}]
It suffices to note that according to~\cite{OSz4}, if $Y$ admits a metric of constant positive curvature, then $Y$ is an $L$-space.
\end{proof}

\section{4-dimensional cobordism and absolute $\mathbb{Q}$-grading}

Let $W$ be a connected compact oriented 4-dimensional cobordism between two connected oriented 3-manifolds $Y_0$ and $Y_1$ such that $\bdry W = -Y_0 \cup Y_1$. Fixing a Spin$^c$ structure $\mathfrak{t}$ on $W$, Ozsv\'ath-Szab\'o~\cite{OSz3} constructed a map $F_{W,\mathfrak{s}}: HF^\circ(Y_0,\mathfrak{t}|_{Y_0}) \to HF^\circ(Y_1,\mathfrak{t}|_{Y_1})$ between Heegaard Floer homology groups by choosing a handle decomposition of $W$, and counting holomorphic triangles. It turns out that $F_{W,\mathfrak{t}}$ is an invariant of $W$, i.e., it is independent of the choice of a handle decomposition of $W$. Throughout this section we fix a Heegaard diagram $(\Sigma,\bm{\alpha},\bm{\beta})$ for $Y_0$ and a handle decomposition of $W$. Let $(\Sigma,\bm{\alpha},\bm{\gamma})$ be the associated Heegaard diagram for $Y_1$ as constructed in~\cite{OSz3}. We consider the associated chain map $F_{W,\mathfrak{t}}:\widehat{CF}(\bm{\alpha},\bm{\beta},\mathfrak{t}|_{Y_0}) \to \widehat{CF}(\bm{\alpha},\bm{\gamma},\mathfrak{t}|_{Y_1})$.

Observe that $F_{W,\mathfrak{t}}:\widehat{CF}(\bm{\alpha},\bm{\beta},\mathfrak{t}|_{Y_0}) \to \widehat{CF}(\bm{\alpha},\bm{\gamma},\mathfrak{t}|_{Y_1})$ is a linear map between graded vector spaces. However, according to Theorem~\ref{mainThm}(a), $\widehat{CF}(\bm{\alpha},\bm{\beta},\mathfrak{t}|_{Y_i})$ is graded by the set of homotopy classes of oriented 2-plane fields $\mathcal{P}(Y_i)$, $i=0,1$, so it is not possible to define an integer degree of $F_{W,\mathfrak{t}}$. There is a weaker notion which is applicable here. Namely, let $W: Y_0 \to Y_1$ be a cobordism and $\xi_i$ be an oriented 2-plane field on $Y_i$, for $i=0,1$. We say $\xi_0 \sim_W \xi_1$ if and only if there exists an almost complex structure $J$ on $W$ such that $[\xi_i]=[TY_i \cap J(TY_i)]$, for $i=0,1$, as homotopy classes of oriented 2-plane fields.

The main goal of this section is to prove Theorem~\ref{mainThm}(d) on the chain level, which we formalize in the following theorem for the reader's convenience.

\begin{thm} \label{CobThm}
Let $W: Y_0 \to Y_1$ be a compact oriented cobordism with a fixed handle decomposition, $\mathfrak{t} \in$ {\em Spin}$^c(W)$ a {\em Spin}$^c$ structure on $W$, and $F_{W,\mathfrak{t}}:\widehat{CF}(\bm{\alpha},\bm{\beta},\mathfrak{t}|_{Y_0}) \to \widehat{CF}(\bm{\alpha},\bm{\gamma},\mathfrak{t}|_{Y_1})$ the associated cobordism map as discussed above. Then $\widetilde{\text{gr}}(\mathbf{x}) \sim_W \widetilde{\text{gr}}(\mathbf{y})$ for any homogeneous generator $\mathbf{x} \in \mathbb{T}_\alpha \cap \mathbb{T}_\beta$ in $\widehat{CF}(\bm{\alpha},\bm{\beta},\mathfrak{t}|_{Y_0})$, and any homogeneous summand $\mathbf{y}$ of $F_{W,\mathfrak{t}}(\mathbf{x})$.
\end{thm}

Before we give the proof of Theorem~\ref{CobThm}, we take a step back and look at the Heegaard Floer homology $HF^\circ(Y,\mathfrak{s})$ for a torsion Spin$^c$ structure $\mathfrak{s}$. By~\cite{OSz3}, there is an absolute $\mathbb{Q}$-grading of $HF^\circ(Y,\mathfrak{s})$ which lifts the relative $\mathbb{Z}$-grading. We shall see that our construction indeed generalizes their absolute $\mathbb{Q}$-grading. To do so, recall the following construction due to R. Gompf~\cite{Gom}. Let $\xi$ be an oriented 2-plane field on a closed, oriented 3-manifold $Y$. Then there exists a compact, almost complex 4-manifold $(X,J)$ whose {\em almost-complex boundary} is $(Y,\xi)$, i.e. $Y=\bdry X$ (as oriented manifolds) and $\xi=TY \cap J(TY)$ with the complex orientation. If $c_1(\xi)$ is a torsion class, then let $\theta(\xi)= (PD~c_1(X))^2-2\chi(X)-3\sigma(X) \in \mathbb{Q}$, where $\chi$ is the Euler characteristic and $\sigma$ is the signature. Observe that $\theta(\xi)$ is independent of the choice of the capping almost complex 4-manifold $(X,J)$ because the quantity $(PD~c_1(X))^2-2\chi(X)-3\sigma(X)$ vanishes for a closed $X$.

Let $\mathfrak{s} \in \textrm{Spin}^c(Y)$ be a Spin$^c$ structure such that $c_1(\mathfrak{s})$ is a torsion class, and let $\mathfrak{U}$ be the set of homogeneous elements in $\widehat{CF}(Y,\mathfrak{s})$. We define an absolute grading function $\widetilde{\text{gr}}_0: \mathfrak{U} \to \mathbb{Q}$ by $\widetilde{\text{gr}}_0(\mathbf{x})=(2+\theta(\widetilde{\text{gr}}(\mathbf{x})))/4 \in \mathbb{Q}$ for any $\mathbf{x} \in \mathfrak{U}$. This induces an absolute grading function on $CF^\infty(Y,\mathfrak{s})$ by $\widetilde{\text{gr}}_0([\mathbf{x},i])=2i+\widetilde{\text{gr}}_0(\mathbf{x})$, and hence on the sub- and quotient-complexes $CF^-(Y,\mathfrak{s})$ and $CF^+(Y,\mathfrak{s})$.

For reader's convenience, we recall the following theorem/definition of the absolute $\mathbb{Q}$-grading due to Ozsv\'ath-Szab\'o~\cite{OSz3}.

\begin{thm}[Ozsv\'ath-Szab\'o]
There exists an absolute grading function $\overline{gr}:\mathfrak{U} \to \mathbb{Q}$ satisfying the following properties:
\begin{enumerate}
    \item{The homogeneous elements of least grading in $\widehat{HF}(S^3, \mathfrak{s}_0)$ have absolute grading zero.}
    \item{The absolute grading lifts the relative grading, in the sense that if $\mathbf{x}, \mathbf{y} \in \mathfrak{U}$, then $\overline{gr}(\mathbf{x},\mathbf{y})=\overline{gr}(\mathbf{x})-\overline{gr}(\mathbf{y})$.}
    \item{If $W$ is a cobordism from $Y_0$ to $Y_1$ endowed with a Spin$^c$ structure $\mathfrak{t}$ whose restriction to $Y_i$ is torsion for $i = 0, 1$, then
        \begin{equation*}
            \overline{\textrm{gr}}(F_{W,\mathfrak{t}}(\mathbf{x})) - \overline{\textrm{gr}}(\mathbf{x})= \frac{(PD~c_1(\mathfrak{t}))^2-2\chi(W)-3\sigma(W)}{4}
        \end{equation*}
        for any $\mathbf{x} \in \mathfrak{U}$.}
\end{enumerate}
\end{thm}

We have the following corollary:

\begin{cor}\label{ColQ}
The function $\widetilde{\text{gr}}_0$ described above defines an absolute $\mathbb{Q}$-grading for $HF^\circ(Y,\mathfrak{s})$, which coincides with the absolute $\mathbb{Q}$-grading $\overline{\textrm{gr}}$ defined above.
\end{cor}

\begin{proof}
We use the Pontryagin-Thom construction. By fixing a trivialization of $TY$, the homotopy classes of oriented 2-plane fields on $Y$ are 1-1 correspondent to the framed cobordism classes of framed links in $Y$. The first assertion of the corollary follows from Theorem~\ref{mainThm}(a) and the observation that adding a right-handed full twist to $\xi$ is equivalent to decreasing $\theta(\xi)$ by 4.

It follows from the proof of Theorem~\ref{CobThm} that if $\mathfrak{t}$ be a Spin$^c$ structure on $W$ whose restriction to $Y_i$ is torsion, for $i=0,1$, then $F_{W,\mathfrak{t}}(\mathbf{x})$ is homogeneous for every homogeneous element $\xx\in \mathfrak{U}$. Since we have shown in Theorem~\ref{ThmRel} that our absolute grading $\widetilde{\text{gr}}$ refines the relative grading, in order to show that $\widetilde{\text{gr}}_0$ coincides with the absolute $\mathbb{Q}$-grading defined in~\cite{OSz3}, it suffices to verify the following two conditions:
\begin{enumerate}
    \item{(Normalization) For the standard contact 3-sphere $(S^3,\xi_{std})$, $\widetilde{\text{gr}}_0(c(\xi_{std}))=0$.}
    \item{(Cobordism formula) Let $W:Y_0 \to Y_1$ be a cobordism, and $\mathfrak{t}$ be a Spin$^c$ structure on $W$ whose restriction to $Y_i$ is torsion, $i=0,1$. Then
            \begin{equation*}
            \widetilde{\text{gr}}_0(F_{W,\mathfrak{t}}(\mathbf{x})) - \widetilde{\text{gr}}_0(\mathbf{x})= \frac{(PD~c_1(\mathfrak{t}))^2-2\chi(W)-3\sigma(W)}{4}
            \end{equation*}
        for any homogeneous $\mathbf{x} \in \mathfrak{U}$.
        }
\end{enumerate}

To prove (1), note that it follows from the fact that $(S^3,\xi_{std})$ is the almost complex boundary of the standard unit 4-ball $B^4 \subset \mathbb{C}^2$.

To prove (2), let $(X,J)$ be an almost complex 4-manifold with almost complex boundary $(Y_0,\widetilde{\text{gr}}(\mathbf{x}))$. By Theorem~\ref{CobThm}, there exists an almost complex structure $J'$ on $W$ such that both $\widetilde{\text{gr}}(\mathbf{x})$ and $\widetilde{\text{gr}}(F_{W,\mathfrak{t}}(\mathbf{x}))$ are $J'$-invariant with the complex orientation. We obtain a new almost complex 4-manifold with almost complex boundary $(X \cup_{Y_0} W,\widetilde{\text{gr}}(F_{W,\mathfrak{t}}(\mathbf{x})))$ by gluing $(X,J)$ and $(W,J')$ along $Y_0$. Recall the following theorem on the signature of 4-manifolds due to Novikov:

\begin{thm}[Novikov]
Let $M$ be an oriented 4-manifold obtained by gluing two 4-manifolds $M_1$ and $M_2$ along some components of their boundaries. Then the signature is additive:
\begin{equation*}
    \sigma(M)=\sigma(M_1)+\sigma(M_2).
\end{equation*}
\end{thm}

We therefore calculate as follows:
\begin{align*}
\widetilde{\text{gr}}_0(F_{W,\mathfrak{t}}(\mathbf{x})) - \widetilde{\text{gr}}_0(\mathbf{x}) &=\frac{\theta(\widetilde{\text{gr}}(F_{W,\mathfrak{t}}(\mathbf{x})))-\theta(\widetilde{\text{gr}}(\mathbf{x}))}{4} \\
&=\frac{(PD~c_1(W,J'))^2-2\chi(W)-3\sigma(W)}{4} \\
&=\frac{(PD~c_1(\mathfrak{t}))^2-2\chi(W)-3\sigma(W)}{4},
\end{align*}

This finishes the proof of the second assertion of the corollary.
\end{proof}

The proof of Theorem~\ref{CobThm} occupies the rest of this section. We shall follow the construction of $F_{W,\mathfrak{t}}$ given in~\cite{OSz3}.

\begin{proof}[Proof of Theorem~\ref{CobThm}]
We fix a handle decomposition of $W$, and study the 2-handle attachments and 1- and 3-handle attachments in $W$ separately.

\s\n
\textsc{case 1.} Suppose $W$ is given by 2-handle attachments along a framed link $L \subset Y_0$. Let $\Delta$ denote the two-simplex, with vertices $v_\alpha, v_\beta, v_\gamma$ labeled clockwise, and let $e_i$ denote
the edge $v_j$ to $v_k$, where $\{i,j,k\} = \{\alpha, \beta, \gamma\}$. Recall that given a Heegaard triple $(\Sigma,\bm{\alpha},\bm{\beta},\bm{\gamma})$, one can associate to it a 4-manifold

\begin{equation} \label{4mfdwCorner}
W_{\alpha,\beta,\gamma}=\frac{(\Delta\times\Sigma)\coprod(e_\alpha\times U_\alpha)\coprod(e_\beta\times U_\beta)\coprod(e_\gamma\times U_\gamma)}{(e_\alpha\times\Sigma)\sim(e_\alpha\times\bdry U_\alpha), (e_\beta\times\Sigma)\sim(e_\beta\times\bdry U_\beta), (e_\gamma\times\Sigma)\sim(e_\gamma\times\bdry U_\gamma)}
\end{equation}

\n
where $U_\alpha$ (resp. $U_\beta$, $U_\gamma$) is the handlebody determined by the $\bm\alpha$ (resp. $\bm\beta$, $\bm\gamma$) curves. Let $Y_{\alpha,\beta} = U_\alpha \cup U_\beta$, $Y_{\beta,\gamma} = U_\beta \cap U_\gamma$, and $Y_{\alpha,\gamma} = U_\alpha \cup U_\gamma$ be the 3-manifolds obtained by gluing the $\alpha$-, $\beta$- and $\gamma$-handlebodies along $\Sigma$ in pairs. After smoothing the corners, we have
\begin{equation*}
\bdry W_{\alpha,\beta,\gamma} = -Y_{\alpha,\beta} -Y_{\beta,\gamma} + Y_{\alpha,\gamma}
\end{equation*}
as oriented manifolds. See Figure~\ref{cobordism}.

\begin{figure}[ht]
    \begin{overpic}[scale=.27]{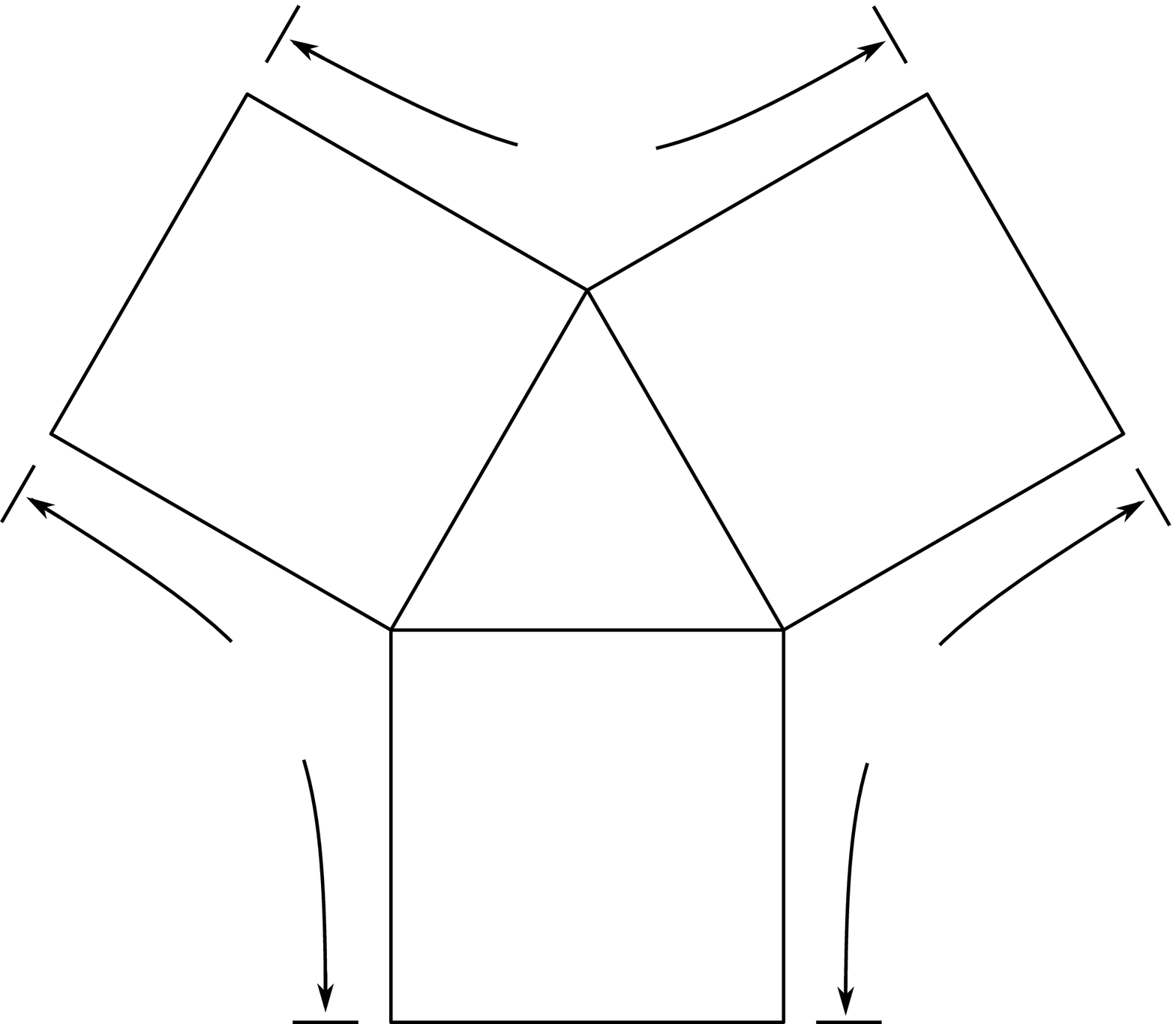}
    \put(48,29){$e_\alpha$}
    \put(34,49){$e_\beta$}
    \put(59.5,49){$e_\gamma$}
    \put(18,26){\small $Y_{\alpha,\beta}$}
    \put(73,26){\small $Y_{\alpha,\gamma}$}
    \put(45,73){\small $Y_{\beta,\gamma}$}
    \end{overpic}
    \caption{The 4-manifold $W_{\alpha,\beta,\gamma}$ associated with a Heegaard triple $(\Sigma,\bm\alpha,\bm\beta,\bm\gamma)$.}
    \label{cobordism}
\end{figure}

According to~\cite{OSz3}, if $W$ is obtained by attaching 2-handles along a framed link $L$, then there exists a triple Heegaard diagram $(\Sigma,\bm\alpha,\bm\beta,\bm\gamma,z)$ such that $Y_{\alpha,\beta}=Y_0$, $Y_{\beta,\gamma}={\#}^n(S^1\times S^2)$ for some $n \geq 1$, and $Y_{\alpha,\gamma}=Y_1$. Moreover, after filling in the boundary component $Y_{\beta,\gamma}$ by the boundary connected sum ${\#}^n_b(S^1\times B^3)$, we obtain the original cobordism $W$. Fix a Spin$^c$ structure $\mathfrak{t}$ on $W$ with $\mathfrak{s}_i=\mathfrak{t}|_{Y_i}$, $i=0,1$. Let $\Theta \in \widehat{CF}({\#}^n(S^1\times S^2))$ be the top dimensional generator and let $\mathbf{x} \in \mathbb{T}_\alpha \cap \mathbb{T}_\beta$. By definition, the image of $\xx$ under the cobordism map $F_{W,\mathfrak{t}}: \widehat{CF}(Y_0,\mathfrak{s}_0) \to \widehat{CF}(Y_1,\mathfrak{s}_1)$ is a linear combination of the generators $\yy\in\TT_\alpha\cap\TT_\gamma$ with coefficients being the count of Maslov index 0 holomorphic triangles connecting $\xx$, $\Theta$ and $\yy$. Let $\mathbf{y}$ be a generator appearing in $F_{W,\mathfrak{t}}$ with a nonzero coefficient. We prove the following claim.\\

\begin{figure}[ht]
    \begin{overpic}[scale=.3]{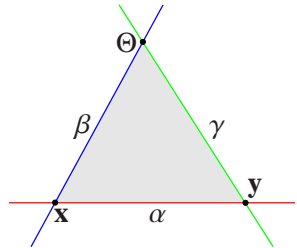}
    \put(16,9){\small$\mathbf{x}$}
    \put(37.5,68){\small$\Theta$}
    \put(83.3,19.2){\small$\mathbf{y}$}
    \put(49,9){\small$\alpha$}
    \put(23.5,40){\small$\beta$}
    \put(70,40){\small$\gamma$}
    \end{overpic}
    \caption{A holomorphic triangle on $\Sigma$ which connects $\mathbf{x}$, $\Theta$, and $\mathbf{y}$.}
    \label{holtriangle}
\end{figure}

\n
{\em Claim}: There exists an almost complex structure $J$ on $W_{\alpha,\beta,\gamma}$ such that $\widetilde{\text{gr}}(\mathbf{x}) \in \mathcal{P}(Y_0)$, $\widetilde{\text{gr}}(\Theta) \in \mathcal{P}({\#}^n(S^1 \times S^2))$, and $\widetilde{\text{gr}}(\mathbf{y}) \in \mathcal{P}(Y_1)$ are all $J$-invariant with the complex orientation.

\begin{proof}[Proof of Claim]
We first assume that $\mathbf{y}$ is the intersection point as shown in Figure~\ref{holtriangle}, which is connected to $\mathbf{x}$ and $\Theta$ by the obvious (embedded) holomorphic triangle. We begin by constructing a 2-plane field on $e_\alpha \times U_\alpha$, and note that the same construction carries over to $e_\beta \times U_\beta$ and $e_\gamma \times U_\gamma$.

For simplicity of notations, we assume $g(\Sigma)=1$, so, for instance, $\mathbf{x} \in \mathbb{T}_\alpha \cap \mathbb{T}_\beta$ is just one point instead of a $g$-tuple of points. The same argument applies to Heegaard surfaces of arbitrary genus without difficulty. Let $V_\alpha$ be the gradient flow on $U_\alpha$ compatible with the $\alpha$-curve so that it is pointing out along $\bdry U_\alpha$. Let $p \in U_\alpha$ be the index 1 critical point of $V_\alpha$ and $w \in U_\alpha$ be the index 0 critical point of $V_\alpha$. Identify the edge $e_\alpha \subset \Delta$ with the subarc of the $\alpha$-curve from $\mathbf{x}$ to $\mathbf{y}$, which is an edge of the holomorphic triangle, such that $v_\gamma$ is identified with $\mathbf{x}$ and $v_\beta$ is identified with $\mathbf{y}$. Abusing notations, we shall not distinguish a point on $e_\alpha$ and the corresponding point on the $\alpha$-curve under the above identification. For any $q \in e_\alpha$, let $\gamma_0$ and $\gamma_1$ be the gradient trajectories which connect $w$ to $z$ and $p$ to $q$ respectively. Let $N(\gamma_i)$ be a tubular neighborhood of $\gamma_i$ as depicted in Figure~\ref{alphahdby}, for $i=0,1$. By restricting the construction of the absolute grading in Section~\ref{Defgr} to $U_\alpha$, we obtain a non-vanishing vector field $V'_{\alpha,q}$ on $U_\alpha$ which depends on the choice of $q \in e_\alpha$ as depicted in Figure~\ref{halfmod}. Thus we have constructed a 2-plane field $\xi_\alpha(q,x) = (V'_{\alpha,q}(x))^{\bot_3}$ on $e_\alpha \times U_\alpha$, for any $q \in e_\alpha$ and $x \in U_\alpha$. Here $\bot_3$ denotes taking the orthogonal complement of $V'_{\alpha,q}$ within $TU_\alpha$.

\begin{figure}[ht]
    \begin{overpic}[scale=.6]{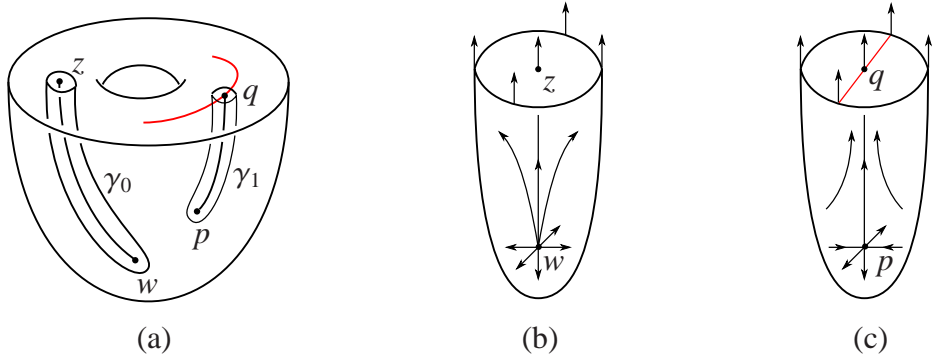}
    \put(10.5,12.5){$\gamma_0$}
    \put(24.5,13){$\gamma_1$}
    \put(14,1){$w$}
    \put(7,25){$z$}
    \put(20,6.5){$p$}
    \put(25.5,22){$q$}
    \put(58,3.5){$w$}
    \put(58,23){$z$}
    \put(94,3.5){$p$}
    \put(93.3,23.7){$q$}
    \put(14,-5){(a)}
    \put(55.7,-5){(b)}
    \put(91.6,-5){(c)}
    \end{overpic}
    \vspace{6mm}
    \caption{(a) The $\alpha$-handlebody $U_\alpha$ and tubular neighborhoods of the gradient trajectories $\gamma_0$ and $\gamma_1$. (b) The gradient vector field $V_\alpha|_{N(\gamma_0)}$ in $N(\gamma_0)$. (c) The gradient vector field $V_\alpha|_{N(\gamma_1)}$ in $N(\gamma_1)$.}
    \label{alphahdby}
\end{figure}

\begin{figure}[ht]
    \begin{overpic}[scale=.4]{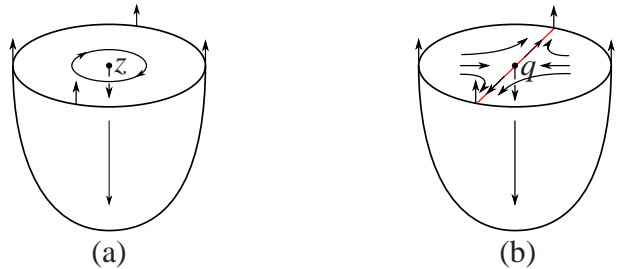}
    \put(17.3,26.5){$z$}
    \put(84.5,26){$q$}
    \put(13.5,-5){(a)}
    \put(81,-5){(b)}
    \end{overpic}
    \vspace{3.5mm}
    \caption{(a) The non-vanishing vector field $V'_{\alpha,q}$ restricted to $N(\gamma_0)$. (b) The non-vanishing vector field $V'_{\alpha,q}$ restricted to $N(\gamma_1)$.}
    \label{halfmod}
\end{figure}

Similarly one constructs 2-plane fields $\xi_\beta$ and $\xi_\gamma$ on $e_\beta \times U_\beta$ and $e_\gamma \times U_\gamma$, respectively. However, note that the boundary component $Y_{\alpha,\beta} = (v_\gamma \times U_\alpha) \cup (v_\gamma \times U_\beta)$ of $W_{\alpha,\beta,\gamma}$ is a 3-manifold with corners, and the 2-plane fields $\xi_\alpha$ and $\xi_\beta$ do not agree along $v_\gamma \times \Sigma$ because they are tangent to the $\alpha$- and $\beta$-handlebodies which intersect each other in an angle. To smooth the corners, we replace the triangle $\Delta$ in (\ref{4mfdwCorner}) with a hexagon $H$ with right corners and attach $\alpha$, $\beta$, and $\gamma$ handles accordingly as depicted in Figure~\ref{smoothing}. In this way we obtain a smooth cobordism which we still denote by $W_{\alpha,\beta,\gamma}: Y_0 \coprod (S^1 \times S^2) \to Y_1$, where $Y_0 = (v_\gamma\times U_\alpha) \cup ([0,1]\times\Sigma) \cup (v_\gamma\times U_\beta)$, $Y_1 = (v_\beta\times U_\alpha) \cup ([0,1]\times\Sigma) \cup (v_\beta\times U_\gamma)$, and $S^1 \times S^2 = (v_\alpha\times U_\beta) \cup ([0,1]\times\Sigma) \cup (v_\alpha\times U_\gamma)$ are smooth 3-manifolds. We construct a 2-plane field $\xi$ on $(e_\alpha \times U_\alpha) \cup (e_\beta \times U_\beta) \cup (e_\gamma \times U_\gamma) \cup \bdry W_{\alpha,\beta,\gamma}$ by extending $\xi_\alpha$, $\xi_\beta$, and $\xi_\gamma$ to the three copies of $[0,1]\times\Sigma$ such that it is translation invariant in the $[0,1]$-direction on each copy. By construction, it is easy to see that $\xi|_{Y_0} \simeq \widetilde{\text{gr}}(\mathbf{x})$, $\xi|_{S^1 \times S^2} \simeq \widetilde{\text{gr}}(\Theta)$, and $\xi|_{Y_1} \simeq \widetilde{\text{gr}}(\mathbf{y})$.

\begin{figure}[ht]
    \begin{overpic}[scale=.32]{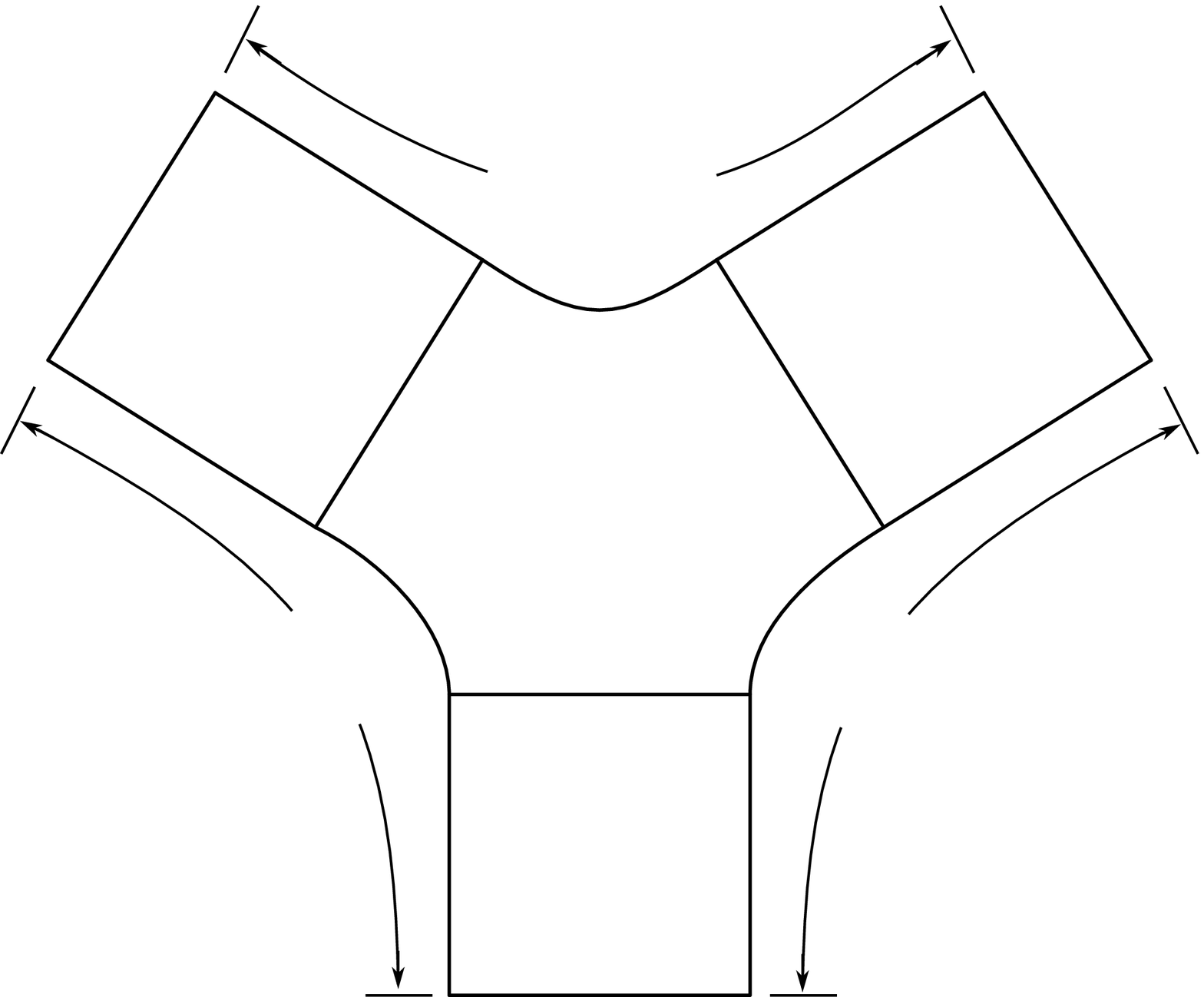}
    \put(48,20){$e_\alpha$}
    \put(26.5,50){$e_\beta$}
    \put(68,50){$e_\gamma$}
    \put(47,40){$H$}
    \put(25,26){$Y_0$}
    \put(70,25.7){$Y_1$}
    \put(42,67.5){\small$S^1 \times S^2$}
    \end{overpic}
    \caption{The smooth cobordism $W_{\alpha,\beta,\gamma}: Y_0 \coprod (S^1 \times S^2) \to Y_1$.}
    \label{smoothing}
\end{figure}

Let $D_1 \subset \Sigma$ be a closed neighborhood of $z$, and $D_2 \subset \Sigma$ be a closed neighborhood of the holomorphic triangle so that the non-vanishing vector field $V'_{i,q}$ is transverse to $T\Sigma$ along $\Sigma\setminus(D_1 \cup D_2)$ for any $i \in \{\alpha,\beta,\gamma\}$, $q \in \bdry\Delta$. We extend $\xi$ to the metric closure of $H\times(\Sigma\setminus(D_1 \cup D_2))$ by letting $\xi(x,y)=T_y\Sigma$ for any $x \in H$, and $y \in \Sigma\setminus(D_1 \cup D_2)$. We construct an almost complex structure $J$ on a subset of $W_{\alpha,\beta,\gamma}$ by asking $\xi$ and $\xi^{\bot_4}$ to be complex line bundles, where $\bot_4$ denotes taking the orthogonal complement in $TW_{\alpha,\beta,\gamma}$. In fact $J$ is defined everywhere on $W_{\alpha,\beta,\gamma}$ except finitely many 4-balls (with corners), namely, $H \times D_1$ and $H \times D_2$. To extend $J$ to the whole $W_{\alpha,\beta,\gamma}$, we round the corners of $\bdry(H \times D_i)$, $i=1,2$, in two steps.\\

\textit{Step 1}. To round the corners of $\bdry H \times D_1$ and $\bdry H \times D_2$ near each vertex of $H$, we first construct a local model for corner-rounding as follows.

Let $(x_1,y_1,x_2,y_2)$ be coordinates on $\mathbb{R}^2 \times \mathbb{R}^2$ with the Euclidean metric. Consider a non-singular vector field
\begin{equation*}
v(x_1,y_1,x_2,y_2)=f(x_2,y_2)\frac{\bdry}{\bdry y_1}+g(x_2,y_2)\frac{\bdry}{\bdry x_2}+h(x_2,y_2)\frac{\bdry}{\bdry y_2}
\end{equation*}
on $\mathbb{R}^2\times\mathbb{R}^2$, namely, $f$, $g$ and $h$ cannot be simultaneously zero. Observe that $v$ is everywhere tangent to $\mathbb{R}^3 \simeq \{(x_1,y_1,x_2,y_2)~|~x_1=\text{constant}\}$. Define $v^{\bot_3}$ to be the pointwise orthogonal complement to $v$ inside $\mathbb{R}^3 \simeq \{(x_1,y_1,x_2,y_2)~|~x_1=\text{constant}\}$. Let $J$ be an almost complex structure on $\mathbb{R}^2\times\mathbb{R}^2$ which preserves the metric and satisfies:
\begin{itemize}
    \item{$J(\frac{\bdry}{\bdry x_1})=\frac{v}{||v||}$,}
    \item{$J(v^{\bot_3})=v^{\bot_3}$.}
\end{itemize}

Let $\mathcal{L}=\{(x_1,0)~|~x_1 \geq 0\} \cup \{(0,y_1)~|~y_1 \geq 0\} \subset \mathbb{R}^2$ be a $L$-shaped broken line with a corner at the origin. We round the corner of $\mathcal{L}$ by considering
\begin{equation*}
\mathcal{L}_r=\{(x_1,0)~|~x_1 \geq 1\} \cup \{(0,y_1)~|~y_1 \geq 1\} \cup \{(x_1-1)^2+(y_1-1)^2=1~|~0 \leq x_1 \leq 1,0 \leq y_1 \leq 1\}.
\end{equation*}

Consider the smooth submanifold $\bar{\mathcal{L}}=\mathcal{L}_r \times \mathbb{R}^2$ in $\mathbb{R}^2\times\mathbb{R}^2$. We compute the complex line distribution $T\bar{\mathcal{L}} \cap J(T\bar{\mathcal{L}})$ on $T\bar{\mathcal{L}}$ with respect to $J$. To do so, identify $\bar{\mathcal{L}}$ with $(-\infty,\infty)\times\mathbb{R}^2$ such that $\{(0,y_1)~|~y_1 \geq 1\}$ is identified with $(-\infty,0]\times\mathbb{R}^2$, $\{(x_1,0)~|~x_1 \geq 1\}$ is identified with $[1,\infty)\times\mathbb{R}^2$, and $\{(x_1-1)^2+(y_1-1)^2=1~|~0 \leq x_1 \leq 1,0 \leq y_1 \leq 1\}$ is identified with $[0,1]\times\mathbb{R}^2$. Let $\phi_t: \mathbb{R}^3 \to \mathbb{R}^3$ be the clockwise rotation about the $x$-axis by $\chi(t)\pi/2$, where $(x,y,z)$ are coordinates on $\mathbb{R}^3$ and
\begin{equation*}
\chi(t)=
    \begin{cases}
    0 & \text{if } t \leq 0,\\
    t & \text{if } 0\leq t \leq 1,\\
    1 & \text{if } t \geq 1.
    \end{cases}
\end{equation*}

\begin{lemma} \label{techclxtangency}
The 2-plane field $T\bar{\mathcal{L}} \cap J(T\bar{\mathcal{L}})$ on $\bar{\mathcal{L}} \simeq (-\infty,\infty)\times\mathbb{R}^2$ is the orthogonal complement of the non-singular vector field $\mu(t,x_2,y_2)=\phi_t(v(x_2,y_2))$.
\end{lemma}

\begin{proof}[Proof of Lemma~\ref{techclxtangency}]
We first compute $J(\frac{\bdry}{\bdry y_1})$ as follows. Note that
\begin{equation*}
v^{\bot_3}=
    \begin{cases}
    span\{\frac{\bdry}{\bdry x_2},\frac{\bdry}{\bdry y_2}\} & \text{if } g=h=0,\\
    span\{g\frac{\bdry}{\bdry y_2}-h\frac{\bdry}{\bdry x_2},\frac{\bdry}{\bdry y_1}-\frac{fg}{\lambda^2}\frac{\bdry}{\bdry x_2}-\frac{fh}{\lambda^2}\frac{\bdry}{\bdry y_2}\} & \text{otherwise}.
    \end{cases}
\end{equation*}
where $\lambda=\sqrt{g^2+h^2}$. Since we assume that $J$ preserves the Euclidean metric, we have
\begin{equation} \label{acsComputation}
    \begin{cases}
    J(\frac{\bdry}{\bdry x_2})=\frac{\bdry}{\bdry y_2} & \text{if } g=h=0,\\
    J(g\frac{\bdry}{\bdry y_2}-h\frac{\bdry}{\bdry x_2})=\frac{\lambda^2}{\sqrt{f^2+\lambda^2}}(\frac{\bdry}{\bdry y_1}-\frac{fg}{\lambda^2}\frac{\bdry}{\bdry x_2}-\frac{fh}{\lambda^2}\frac{\bdry}{\bdry y_2}) & \text{otherwise}.
    \end{cases}
\end{equation}

It follows from (\ref{acsComputation}) and the equation $J(\frac{\bdry}{\bdry x_1})=\frac{v}{||v||}$ that
\begin{equation*}
    J\Big(\frac{\bdry}{\bdry y_1}\Big)=\frac{1}{\sqrt{f^2+\lambda^2}}\Big(-f\frac{\bdry}{\bdry x_1}-g\frac{\bdry}{\bdry y_2}+h\frac{\bdry}{\bdry x_2}\Big).
\end{equation*}

It is easy to see that $T\bar{\mathcal{L}} \cap J(T\bar{\mathcal{L}})$ restricted to $\{t\}\times\mathbb{R}^2$, $t \geq 1$, is the orthogonal complement of $J(\frac{\bdry}{\bdry y_1})=\mu(1,\cdot)$ up to positive rescaling within $T\bar{\mathcal{L}}$. Moreover observe that $T\bar{\mathcal{L}} \cap J(T\bar{\mathcal{L}})$ restricted to $\{t\}\times\mathbb{R}^2$, for $0 \leq t \leq 1$, is the orthogonal complement of $J(t\frac{\bdry}{\bdry y_1}+(1-t)\frac{\bdry}{\bdry x_1})$, which is exactly $\mu(t,\cdot)$ up to positive rescaling.
\end{proof}

Without loss of generality, let $q$ be a vertex of $H$ whose adjacent edges are $e_\alpha$ and $[0,1]$, where $[0,1]$ is an edge of $H$ connecting $\alpha$- and $\beta$-handlebodies. Take a small neighborhood $N(q)$ of $q$ in $H$. Identify $N(q)$ with a small neighborhood of the origin in $\mathbb{R}^2$ restricted to the first quadrant such that $e_\alpha\cup[0,1]$ is identified with $\mathcal{L}$. We can further assume that $J$ is defined on $N(q) \times D_i$ by taking $N(q)$ sufficiently small, and that it is invariant under translation in any direction tangent to $N(q)$. Hence we can apply Lemma~\ref{techclxtangency} to compute the complex line distribution on $\mathcal{L}_r \times D_i \subset N(q) \times D_i$, $i=1,2$, with respect to $J$. By rounding all the corners of $H$ and applying Lemma~\ref{techclxtangency}, we conclude that:

\begin{enumerate}
    \item{The complex line distribution $T(\bdry H \times D_1) \cap JT(\bdry H \times D_1)$ on $\bdry H \times D_1$ is, up to homotopy relative to the boundary, the orthogonal complement of the non-singular vector field $v_1$, where $v_1|_{\{p\} \times D_1}$ is shown on Figure~\ref{CornerRounding}(a). In particular $v_1$ is defined to be invariant in the direction of $\bdry H$.}
    \item{Let $\theta \in [0,2\pi)$ be the coordinate on $\bdry H$ with the boundary orientation and $\psi: \bdry H \times D_2 \to \bdry H \times D_2$ be a diffeomorphism defined by $\psi(\theta,z)=(\theta,e^{i\theta}z)$. The complex line distribution $T(\bdry H \times D_2) \cap JT(\bdry H \times D_2)$ on $\bdry H \times D_2$ is, up to homotopy relative to the boundary, the orthogonal complement of the non-singular vector field $v_2=\psi_\ast(v'_2)$, where $v'_2$ is invariant in the direction of $\bdry H$ and its restriction to $p \times D_2$, $p \in \bdry H$, is shown on Figure~\ref{CornerRounding}(b).}
\end{enumerate}

\begin{figure}[ht]
    \begin{overpic}[scale=.47]{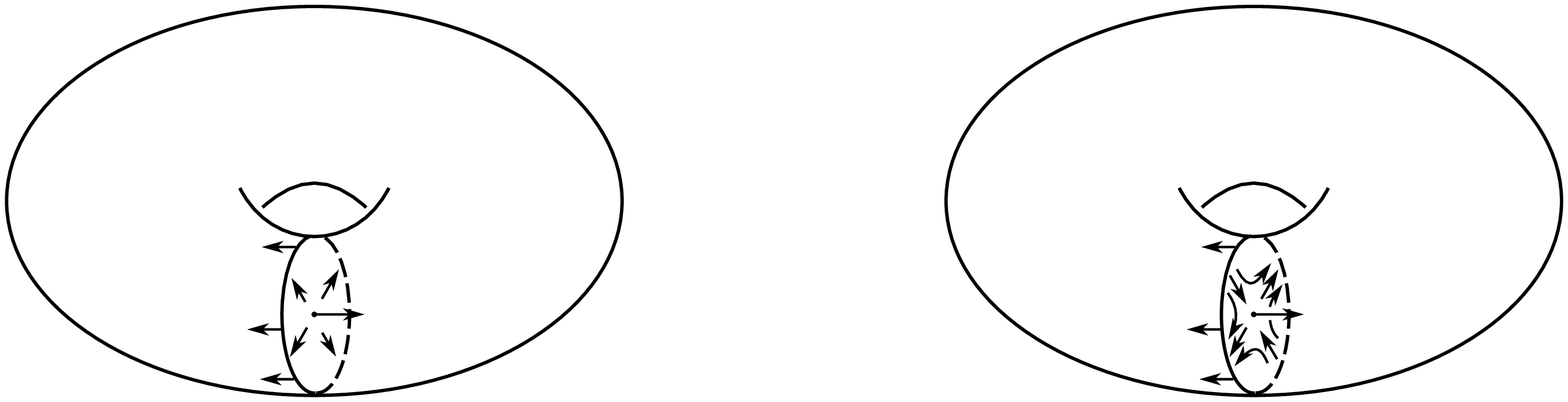}
    \put(14,-5){$\bdry H \times D_1$}
    \put(74,-5){$\bdry H \times D_2$}
    \put(18,-10){(a)}
    \put(78,-10){(b)}
    \end{overpic}
    \vspace{1cm}
    \caption{}
    \label{CornerRounding}
\end{figure}

\textit{Step 2}. Now we round the corners of $\bdry(H \times D_i)=(\bdry H \times D_i) \cup (H \times \bdry D_i)$, which is the union of two solid tori meeting each other orthogonally. Note that the 2-plane field $T(H \times \bdry D_i) \cap JT(H \times \bdry D_i)$ on $H \times \bdry D_i$ is everywhere tangent to $H$ by our choice of $D_i \subset \Sigma$, for $i=1,2$. Abusing notations, we still denote by $\bdry(H \times D_i)$ the smooth 3-sphere obtained by rounding the corners in the standard way. Let $\xi_i$ denote $T(\bdry(H \times D_i)) \cap JT(\bdry(H \times D_i))$, for $i=1,2$. So $\xi_1$ and $\xi_2$ are oriented 2-plane fields. Using the Pontryagin-Thom construction, we see that $\xi_1$ is homotopic to the negative standard contact structure on $S^3$, while $\xi_2$ is homotopic to the positive standard contact structure on $S^3$. Embed $H \times D_i = B^4 \subset \mathbb{C}^2$ such that $H$ and $D_i$ are contained in orthogonal complex planes respectively. Let
\begin{equation*}
    J_0=\left(
    \begin{array}{ccc}
    i & 0 \\
    0 & i
    \end{array} \right),~~~~~~~~~~~~~
    J'_0=\left(
    \begin{array}{ccc}
    i & 0 \\
    0 & -i
    \end{array} \right)
\end{equation*}
be complex structures on $\mathbb{C}^2$. Then it is standard to check that $\xi_1 \simeq TS^3 \cap J'_0 TS^3$ and $\xi_2 \simeq TS^3 \cap J_0 TS^3$ as oriented 2-plane fields, where $S^3 = \bdry B^4 \subset \mathbb{C}^2$. Hence we can extend $J$ to the whole $W_{\alpha,\beta,\gamma}$ satisfying all the desired properties.

Now we turn to the general case. Let $\mathbf{y}' \in \mathbb{T}_\alpha \cap \mathbb{T}_\gamma$ be another intersection point in $F_{W,\mathfrak{t}}$, i.e. there exists a holomorphic triangle $\psi' \in \pi_2(\mathbf{x},\Theta,\mathbf{y}')$ such that the Maslov index $\mu(\psi')=0$. Let $\mathbf{y} \in F_{W,\mathfrak{t}}(\mathbf{x})$ be the intersection point as shown in Figure~\ref{holtriangle} and $\psi \in \pi_2(\mathbf{x},\Theta,\mathbf{y})$ be the obvious holomorphic triangle of Maslov index $\mu(\psi)=0$. Since $\psi$ and $\psi'$ induces the same Spin$^c$ structure $\mathfrak{t}$ on $W$, we have $\psi'=\psi+\phi_1+\phi_2+\phi_3$ for $\phi_1\in\pi_2(\mathbf{x},\mathbf{x})$, $\phi_2\in\pi_2(\Theta,\Theta)$, and $\phi_3\in\pi_2(\mathbf{y},\mathbf{y}')$. This implies
\begin{equation*}
    \mu(\psi')=\mu(\psi)+\mu(\phi_1)+\mu(\phi_2)+\mu(\phi_3).
\end{equation*}
Therefore
\begin{equation*}
    \mu(\phi_1)-2n_z(\phi_1)=-(\mu(\phi_3)-2n_z(\phi_3)),
\end{equation*}
because $\mu(\psi)=\mu(\psi')=n_z(\psi)=n_z(\psi')=\mu(\phi_2)-2n_z(\phi_2)=0$. Since we have shown that there exists an almost complex structure $J$ on $W_{\alpha,\beta,\gamma}$ such that $\widetilde{\textrm{gr}}(\mathbf{x}) \in \mathcal{P}(Y_0)$, $\widetilde{\textrm{gr}}(\mathbf{y}) \in \mathcal{P}(Y_1)$ and $\widetilde{\textrm{gr}}(\Theta) \in \mathcal{P}(\#^n(S^1 \times S^2))$ are all $J$-invariant with the complex orientation, it is easy to show that there exists another almost complex structure $J'$ on $W_{\alpha,\beta,\gamma}$ such that $\widetilde{\textrm{gr}}(\mathbf{x})+\mu(\phi_1)-2n_z(\phi_1)$, $\widetilde{\textrm{gr}}(\mathbf{y})-(\mu(\phi_3)-2n_z(\phi_3))$, and $\widetilde{\textrm{gr}}(\Theta)$ are all $J'$-invariant with the complex orientation. Here we are using the $\Z$-action as explained in Remark~\ref{RmkAct}. Now it remains to observe that $\widetilde{\textrm{gr}}(\mathbf{x})=\widetilde{\textrm{gr}}(\mathbf{x})+\mu(\phi_1) -2n_z(\phi_1)\in \mathcal{P}(Y_0)$ since $\mu(\phi_1)-2n_z(\phi_1)$ is an integral multiple of the divisibility of $c_1(\widetilde{\textrm{gr}}(\mathbf{x})) \in H^2(Y_0;\mathbb{Z})$, and that $$\widetilde{\textrm{gr}}(\mathbf{y}')=\widetilde{\textrm{gr}}(\mathbf{y})-\textrm{gr}(\yy,\yy')=\widetilde{\textrm{gr}}(\mathbf{y})-(\mu(\phi_3)-2n_z(\phi_3)).$$
\end{proof}

It remains to show that $J$ can be extended to $W$. Recall that $W = W_{\alpha,\beta,\gamma} \cup {\#}^n_b (S^1 \times B^3)$. We need to show that there exists an almost complex structure on ${\#}^n_b (S^1 \times B^3)$ such that its restriction to ${\#}^n (S^1 \times S^2) = \bdry ({\#}^n_b (S^1 \times B^3))$ coincides with $J|_{{\#}^n (S^1 \times S^2)}$. Note that $[\Theta] \in \widehat{HF}(-{{\#}}^n (S^1 \times S^2))$ defines the contact invariant of the standard contact structure on ${\#}^n (S^1 \times S^2)$, which is holomorphically fillable. Hence the conclusion follows immediately from Theorem~\ref{mainThm}(b). We finish the proof of Case 1.

\s\n
\textsc{Case 2}. Suppose $W$ is given by attaching 1- and 3-handles. By duality, it suffices to consider the case that $W$ consists of 1-handle attachments. Let $(\Sigma,\bm{\alpha},\bm{\beta},z)$ be a Heegaard diagram of $Y_0$ and $(\Sigma_0,\bm{\alpha}_0,\bm{\beta}_0,z_0)$ a standard Heegaard diagram of ${\#}^n (S^1 \times S^2)$. We obtain a Heegaard diagram $(\Sigma',\bm{\alpha}',\bm{\beta}',z') = (\Sigma,\bm{\alpha},\bm{\beta},z) {\#} (\Sigma_0,\bm{\alpha}_0,\bm{\beta}_0,z_0)$ of $Y_1$. There is an associated map between the Heegaard Floer homology groups
\begin{equation*}
    F_{W,\mathfrak{t}}: \widehat{CF}(\Sigma,\bm{\alpha},\bm{\beta},z,\mathfrak{t}|_{Y_0}) \to \widehat{CF}(\Sigma',\bm{\alpha}',\bm{\beta}',z',\mathfrak{t}|_{Y_1})
\end{equation*}
which is induced by $\displaystyle{F_{W,\mathfrak{t}}(\mathbf{x}) = \mathbf{x} \otimes \Theta}$, where $\mathbf{x} \in \mathbb{T}_\alpha \cap \mathbb{T}_\beta$ is a generator in the Spin$^c$ structure $\mathfrak{t}|_{Y_0}$, and $\Theta \in \widehat{CF}({\#}^n (S^1 \times S^2))$ is the top dimensional generator. Now the existence of an almost complex structure $J$ on $W$ with desired properties follows from Theorem~\ref{mainThm}(b) and the fact that the standard contact structure on ${\#}^n (S^1 \times S^2)$ is fillable by $({\#}^n_b (S^1 \times B^3),J')$ for some almost complex structure $J'$. So Case 2 is also proved.
\end{proof}

\section{The invariance under Heegaard moves}

Our aim for this section is to show that the absolute grading is an invariant of the 3-manifold. That means that if we have two different Heegaard diagrams for the same 3-manifold, then the absolute grading is preserved under the isomorphism between the Floer homologies defined in \cite{OSz2}. It is shown in \cite{OSz2} that any two Heegaard diagrams for the same manifold differ by a sequence of Heegaard moves, i.e. isotopies, handleslides, stabilizations and destabilizations. Every Heegaard move gives rise to a chain map between the Floer complexes, which induces an isomorphism in homology. It is easy to see that these chain maps take homogeneous elements to homogeneous elements. We will show the following theorem.

\begin{thm}\label{inv}
 Let $(\Sigma,\ba,\bb,z)$ be a Heegaard diagram for $Y$ and $(\Sigma',\ba',\bb',z')$ a Heegaard diagram obtained by a Heegaard move from $(\Sigma,\ba,\bb,z)$. Let $\Gamma:~\widehat{CF}(\Sigma,\ba,\bb,z)\to\widehat{CF}(\Sigma',\ba',\bb',z')$ be the chain map defined in~\cite{OSz2}. If $\xx\in\TT_{\ba}\cap\TT_{\bb}$, then $\widetilde{\text{gr}}(\xx)=\widetilde{\text{gr}}(\Gamma(\xx))$.
\end{thm}

\begin{rmk}Theorem \ref{inv} gives the invariance we wanted and implies that the following decomposition is independent of the Heegaard diagram.
$$
 \widehat{HF}(Y;\fs)=\bigoplus_{\rho\in\mathcal{P}(Y,\fs)} \widehat{HF}_{\rho}(Y;\fs),
$$
\end{rmk}

To prove Theorem \ref{inv}, we will consider each type of Heegaard move at a time.

\subsection{Isotopies}
Let $(\Sigma,\ba,\bb,z)$ be a Heegaard diagram for $Y$ and let $\ba'$ be given by moving $\alpha_1$ to $\alpha_1'$ by a Hamiltonian isotopy without passing through $z$. Then there is a continuation map $\Gamma:~\widehat{CF}(\Sigma,\ba,\bb,z)\rightarrow\widehat{CF}(\Sigma,\ba',\bb,z)$ defined by counting Maslov index 0 holomorphic disks with dynamic boundary conditions, as defined in \cite{OSz2}. If this isotopy does not create or destroy intersections between $\alpha$ and $\beta$ curves, then it corresponds to isotoping the Morse function without introducing or removing any critical point. In this case it is clear that $\Gamma$ is an isomorphism and that it preserves the absolute grading.

\begin{center}\begin{figure}[ht]
    \begin{overpic}[scale=0.7]{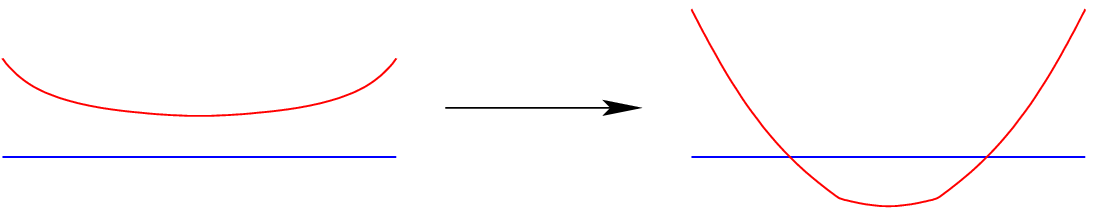}
    \end{overpic}
    \caption{}
    \label{fing}
\end{figure}\end{center}

A finger move is a Hamiltonian isotopy that creates a canceling pair of intersections, as shown in Figure \ref{fing}. We only need to show that $\Gamma$ is invariant when the isotopy introduces or eliminates one finger move and the general isotopy invariance follows from that. First assume that $\alpha_1'$ is obtained from $\alpha_1$ by introducing one finger move. Let $\xx=(x_1,\dots,x_g)\in \TT_{\ba}\cap\TT_{\bb}$, where $x_i\in\alpha_i\cap\beta_{\sigma(i)}$, for some permutation $\sigma$. Then $x_1$ is moved to a point $x_1'\in\alpha_1'\cap\beta_{\sigma(1)}$. We note that $x_1'$ is never one of the two new intersection points. It is easy to see an index 0 holomorphic disk from $x_1$ to $x_1'$, which is actually just a flow line along $\beta_{\sigma(1)}$. So if we take $\xx'=(x_1',x_2,\dots,x_g)$, then $\xx'$ is one of the terms in $\Gamma(\xx)$. It is easy to see that $\widetilde{\text{gr}}(\xx)=\widetilde{\text{gr}}(\xx')$. Therefore $\Gamma$ preserves the absolute grading. Now we assume that $\alpha_1'$ is obtained from $\alpha_1$ by eliminating a finger move. It remains to see what happens when $x_1$ is one of the two points that disappears. So we assume that $x_1$ is one of those two points, such that $\xx=
(x_1,\dots,x_g)\in\widehat{CF}(\Sigma,\ba,\bb,z)$. If $\Gamma(\xx)=0$, then there is nothing to prove. Assume that $\Gamma(\xx)\neq 0$. So we can take a term $\xx'$ in $\Gamma(\xx)$. Then since we only isotoped $\alpha_1$, none of the points $x_i$, for $i>1$, have moved. So we can write $\xx'= (x_1',x_2,\dots,x_g)$, where $x_1'\in \alpha_1'\cap\beta_{\sigma(1)}$. That means that there exists a Maslov index 0 holomorphic disk $\varphi$ from $x_1$ to $x_1'$. Now undoing this isotopy and introducing the finger move again, $x_1'$ corresponds to an intersection $x_1''\in \alpha_1\cap \beta_{\sigma(1)}$ and there is a Maslov index zero holomorphic disk $\psi$ from $x_1'$ to $x_1''$. We now observe that the composition $\varphi *\psi$ is homotopic to a Whitney disk from $x_1$ to $x_1''$ with stationary boundary conditions, i.e. there exists a Whitney disk from $x_1$ to $x_1''$ with its boundary mapping to $\alpha_1\cup\beta_{\sigma(1)}$. Therefore there is an index zero Whitney disk from $x_1$ to $x_1''$. So, since the absolute grading refines the relative grading in $\widehat{CF}(\Sigma,\ba,\bb,z)$, it follows that $\widetilde{\text{gr}}(\xx)=\widetilde{\text{gr}}(\xx'')$, where $\xx''=(x_1'',x_2,\dots,x_g)$, and hence $\widetilde{\text{gr}}(\xx)=\widetilde{\text{gr}}(\xx')$. That implies that $\Gamma$ preserves the absolute grading when a finger move is undone.

\subsection{Handleslides}

Let $(\Sigma,\ba,\bb,z)$ be a Heegaard diagram for $Y$ and let $\beta_1'$ be the closed curve obtained by handlesliding $\beta_1$ over $\beta_2$. Now we define $\bb'=(\beta_1',\beta_2,\dots,\beta_g)$. This handleslide gives rise to a trivial cobordism $W=Y\times[0,1]$, which can also be obtained from the Heegaard triple diagram $(\Sigma,\ba,\bb,\bb')$ by attaching $g$ copies of $S^1\times D^3$, as explained in~\cite{OSz2}. Let $F_W: \widehat{CF}(\Sigma,\ba,\bb,z)\rightarrow\widehat{CF}(\Sigma,\ba,\bb',z)$ be the induced chain map. Then, it follows from Theorem~\ref{mainThm}(c) that $\widetilde{\text{gr}}(\xx)\sim_W~\widetilde{\text{gr}}(F_W(\xx))$. That means that there exists an almost-complex structure $J$ on $W$ such that $[T(Y\times\{0\})\cap J(T(Y\times\{0\}))]=\widetilde{\text{gr}}(\xx)$ and $[T(Y\times\{1\})\cap J(T(Y\times\{1\}))]=\widetilde{\text{gr}}(F_W(\xx))$. Now let $\xi_t=T(Y\times\{t\})\cap J(T(Y\times\{t\}))$, for $0\le t\le 1$. Under the canonical identification $Y\simeq Y\times\{t\}$, $\{\xi_t\}$ gives a homotopy between
$T(Y\times\{0\})\cap J(T(Y\times\{0\}))$ and $T(Y\times\{1\})\cap J(T(Y\times\{1\}))$. Therefore $\widetilde{\text{gr}}(\xx)=\widetilde{\text{gr}}(F_W(\xx))$.

\subsection{Stabilization}
Given a Heegaard diagram $(\Sigma,\ba,\bb,z)$ we stabilize it by taking the connected sum with a two-torus and introducing a new pair of $\alpha$ and $\beta$ curves in this two-torus that intersect at exactly one point. This is equivalent to taking the connect sum of $Y$ with an $S^3$, that is endowed with the standard genus one Heegaard decomposition. We can write $(\Sigma',\ba',\bb',z')$ for the Heegaard diagram of the stabilization. Here $\Sigma'=\Sigma\# E$, for a two-torus $E$, $\ba'=(\alpha_1,\dots,\alpha_g,\alpha_{g+1})$, $\bb'=(\beta_1,\dots,\beta_g,\beta_{g+1})$ and $z'\in\Sigma'$ is naturally associated with $z$, assuming that the connected sum removes a ball from $\Sigma$ that does not contain $z$. Let $w$ be the unique point in $\alpha_{g+1}\cap\beta_{g+1}$.
It is clear that $\Gamma:\widehat{CF}(\Sigma,\ba,\bb,z)\rightarrow\widehat{CF}(\Sigma',\ba',\bb',z')$, which takes $(x_1,\dots,x_g)$ to $(x_1,\dots,x_g,w)$, is an isomorphism. Is is also shown in \cite{OSz2} that this map gives rise to an isomorphism in homology.
We need to show that the absolute grading is invariant under $\Gamma$. Let $\xx=(x_1,\dots,x_g)\in\widehat{CF}(\Sigma,\ba,\bb,z)$. In the definition of $\widetilde{\text{gr}}(\xx)$ we modify a gradient-like vector field in neighborhoods of the flow lines $\gamma_{x_i}$ and $\gamma_0$ to get a nonzero vector field. We can write $$Y\# S^3=(Y\setminus B_{\varepsilon})\cup_{\phi}(S^3\setminus B_{R}),$$
where $B_{\varepsilon}$ is a small ball, $B_{R}$ is a large ball and $\phi:\partial B_{\varepsilon}\rightarrow\partial B_R$ is a diffeomorphism. We can see the same neighborhoods $N(\gamma_{x_i})\subset Y$ and $N(\gamma_0)\subset Y$ in $Y\# S^3$. Now we take a gradient-like vector field $v$ for a Morse function compatible with $(\Sigma',\ba',\bb,z')$. The definition of $\widetilde{\text{gr}}(\Gamma(\xx))$ clearly implies that the vector field $w_{\Gamma(\xx)}$ is homotopic to $w_{\xx}$ in $Y\setminus B_{\varepsilon}$. So it remains to show that $w_{\xx}$ and $w_{\Gamma(\xx)}$ are also homotopic in $S^3\setminus B_{R}$. We can think of $S^3\setminus B_R$ as a small ball $B_{\delta}$ in $R^3$, where $w_{\xx}$ is very close to being constant with respect to the standard trivialization. We note that $v$ has only two critical points in $B_{\delta}$. It is easy to homotope $w_{\xx}$ in a neighborhood of $B_{\delta}$ so that it coincides with $v$ on $\partial B_{\delta}$. It is also easy to see that after we modify $v$ in $N(\gamma_{x_{g+1}})$, the vector field we obtain is homotopic to $w_{\xx}$ in $B_{\delta}$. That concludes the proof of Theorem~\ref{inv}.

%\nocite{Tur}

\bibliography{mybib}
\bibliographystyle{plain}
\end{document}